\documentclass[3p]{elsarticle}
\usepackage{graphicx}
\usepackage{amsmath}
\usepackage{epstopdf} % needed if you have eps figures in a pdflatex manuscript
\usepackage{mathptmx}      % use Times fonts if available on your TeX system
\usepackage[toc,page,title,titletoc,header]{appendix}
\usepackage{latexsym}
 \usepackage{amsmath}
 \usepackage{subfigure}
 \usepackage{stmaryrd}
 \usepackage{amsfonts,amssymb}
 \usepackage{multicol}
 \usepackage{subfigure}
 \usepackage{multirow,makecell}
 \usepackage{color}
 \usepackage{marvosym}
 \usepackage{ntheorem}
\usepackage[colorlinks]{hyperref}
\usepackage{hypernat}

\graphicspath{{Figs/}}
\newtheorem{lemma}{Lemma}[section]
\newtheorem{theorem}{Theorem}[section]
\newtheorem{example}{Example}[section]
\newdefinition{rmk}{Remark}[section]

\newtheorem*{proof}{Proof}

\begin{document}

\begin{frontmatter}

\title{A conservative difference scheme with optimal pointwise error estimates for two-dimensional space fractional nonlinear Schr\"{o}dinger equations}%

\author[1]{Hongling Hu}
\author[20,2]{Xianlin Jin}
\author[3]{Dongdong He\corref{cor1}}
\cortext[cor1]{Corresponding author.}
\ead{hedongdong@cuhk.edu.cn}
\author[20]{Kejia Pan}
\author[4]{Qifeng Zhang}

\address[1]{Key Laboratory of Computing and Stochastic Mathematics (Ministry of Education), School of Mathematics and Statistics, Hunan Normal University, Changsha, Hunan 410081, China}
\address[20]{School of Mathematics and Statistics, HNP-LAMA, Central South University, Changsha, Hunan 410083, China}
\address[2]{{School of Mathematical Sciences, Peking University, Beijing 100871, China}}
\address[3]{School of Science and Engineering, The Chinese University of Hong Kong, Shenzhen, Shenzhen, Guangdong 518172, China}
\address[4]{Department of Mathematics, Zhejiang Sci-Tech University, {{Hangzhou}}, Zhejiang 310018, China}

\begin{abstract}
In this paper, a linearized semi-implicit finite difference scheme is proposed
for solving the two-dimensional (2D) space  fractional nonlinear Schr\"{o}dinger equation (SFNSE).
{The scheme has the property of mass and energy conservation on the discrete level, with an unconditional stability and a second order accuracy for both time and spatial variables.}
{The main contribution of this paper is an optimal pointwise
error estimate for the 2D SFNSE, which is rigorously established and proved for the first time}.  Moreover,
{a novel technique is proposed for dealing with the nonlinear term in the equation, which plays an essential role in the error estimation.}
Finally, the numerical results confirm well with the theoretical findings.
%Include keywords and mathematical subject classification numbers as needed.


\end{abstract}
\end{frontmatter}
\section{Introduction}\label{int}
Schr\"{o}dinger equation is one of the most important equations in mathematical physics,
which describes non-relativistic quantum mechanical behavior, including modeling the hydrodynamics of
Bose-Einstein condensate \cite{BJM2003}. In the 1960's, Feynman and Hibbs \cite{FH1965}
used path integrals over Brownian paths to derive the standard (non-fractional) Schr\"{o}dinger
 equation. Fractional quantum mechanics is a theory used to discuss quantum phenomena in fractal environments.
Laskin \emph{et al.} \cite{Laskin2000a,Laskin2000b} first successfully attempt to apply the
fractal concept to reformulate the standard Schr\"{o}dinger equation over L\'{e}vy-like quantum mechanical paths
and develop the fractional Schr\"{o}dinger equation, in which the fractional space derivative
replaces the second-order Laplacian in the standard Schr\"{o}dinger equation. Laskin~\cite{Laskin2002}
established the parity conservation law for fractional Schr\"{o}dinger
equation. Guo \emph{et al.}~\cite{GHX2008}
 studied the existence and uniqueness of the global smooth solution to the
period boundary value problem of fractional nonlinear Schr\"{o}dinger equation by using
the energy method. Dong \emph{et al}~\cite{DX2007} investigated the analytic solution of
space fractional Schr\"{o}dinger equation with linear potential, delta-function potential and
Coulomb potential, respectively. There are also numerous papers devoted to
the theoretical property and practical application of fractional Schr\"{o}dinger equations
\cite{B2012,Chen2008,DZ2015,Luch2013,Sim2013}.  In the literature, there are  some researchers
concentrated on finding  the solution of fractional Schr\"{o}dinger equation  by using
the analytical approach.  For instance, Herzallah \emph{et al.} \cite{HG2012} successfully
applied the adomian decomposition method to find the approximated analytical solution of the Schr\"{o}dinger
equation with time and space fractional derivatives. However, exact solutions of the fractional
Schor\"{o}dinger equation are generally not available. Numerical computation
are extremely important for  fractional Schr\"{o}dinger equations. The time fractional
Schr\"{o}dinger equations  have  been numerically studied by many researchers~\cite{BA2015,BZ2017a,HA2016}, we will not describe these literature in detail.

Due to the wide application of the space  fractional Schr\"{o}dinger equation, performing
efficient and accurate numerical simulations for the space  fractional Schr\"{o}dinger
equation plays an essential role in many real applications. For example, Wang \emph{et al.}
\cite{WXY2014,WXY2013} studied the Crank-Nicolson finite difference scheme and implicit
conservative finite difference scheme for the coupled space  fractional one-dimensional (1D) nonlinear
Schr\"{o}dinger equations with the Riesz space fractional derivative, where
the convergent results in $L^2$-norm was obtained. Afterwards, they  investigated the maximum-norm
error analysis for the 1D coupled space  fractional nonlinear Schr\"{o}dinger equations
\cite{WXY2015}. Wang and Huang \cite{WH2015} proposed an energy conservative difference
scheme for the nonlinear fractional Schr\"{o}dinger equations with convergence results in $L^2$-norm.
Huang \emph{et al.} \cite{WHZ2016} obtained the point-wise error estimate of a
conservative difference scheme for the 1D fractional Schr\"{o}dinger equation. Li \emph{et al.}
\cite{LZL2015} investigated the time-space-fractional Schr\"{o}dinger equation using
implicit finite difference scheme. Zhang \emph{et al.} \cite{DZ2016} proposed three Fourier
spectral methods for 1D fractional Schr\"{o}dinger equation. Yang \cite{Yang2016} {proposed a class of linearized energy-conserved difference schemes for nonlinear space-fractional Schr\"{o}dinger equations with second-order convergence in the $L^2$-norm.}
Ran \emph{et al.} \cite{RZ2016}
established a conservative finite difference scheme for strongly coupled 1D fractional
Schr\"{o}dinger equation, where the solvability, stability and convergence in the maximum norm were established. Khaliq \emph{et al.}~\cite{KLF2017} proposed a fourth-order implicit-explicit scheme for 1D
fractional Schr\"{o}dinger equation, where the stability and convergence in $L^2$-norm
were obtained. Li~\emph{et al.}~\cite{LHZ2018} constructed a conservative linearized
Crank-Nicolson Galerkin FEMs for the nonlinear fractional Schr\"{o}dinger equation,
where an unconditional $L^2$-norm error estimates are derived by using a
new error splitting technique. Closely followed by the previous work, Li~\emph{et al.} \cite{LGH2018}
further studied a fast linearized conservative finite element method for the strongly
coupled 1D nonlinear fractional Schr\"{o}dinger equations.

To the authors' best knowledge, there are very few works in the literature for the
numerical methods and numerical analysis for the high dimensional space  fractional
Schr\"{o}dinger equation. {In the small amount of researches}, Zhao \emph{et al.} \cite{ZS2014} and Wang \emph{et al.}
\cite{WH2016} have studied the numerical 2D nonlinear space fractional Schr\"{o}dinger
equation and obtained an  error estimates in $L^2$-norms. By using Strang's splitting technique, Owolabi \emph{et al.}~\cite{OA2016} simulated the space fractional nonlinear Schr\"{o}dinger equation with the Riesz fractional derivative from 1D to 3D with the exponential time-difference method
in time and the Fourier pseudo-spectral method in space.
%Bhrawya and Zaky~\cite{BZ2017b} developed an
%exponentially accurate Jacobi-Gauss-Lobatto collocation method to solve the variable-order
%1D and 2D space fractional Schr\"{o}dinger equations.
Khaliq \emph{et al.} \cite{LKBF2017}
solved the multi-dimensional space-fractional nonlinear Schr\"{o}dinger equations,
where the empirical convergence analysis and calculation of the local truncation error were exhibited.
{After carefully studying the research papers listed above}, we found that the pointwise error estimation for high
dimensional space-fractional Schr\"{o}dinger equation has never been investigated. Furthermore, {the optimal
pointwise error estimate especially the one in the sense of $L^{\infty}$ norm,  which is much more difficult to obtain and has more significant impacts, are still unavailable.}

In this paper, we consider the following 2D SFNSE
\begin{align}
{\bf i}u_t+L_{\alpha}u+|u|^2u=0, \quad (x,y)\in \mathbb{R}^2, \quad 0<t<T,\label{IVP1}
\end{align}
with initial condition
\begin{align}\label{IVP2}
u(x,y,0)=u_0(x,y),\quad (x,y)\in \mathbb{R}^2, \quad 0<t<T,
\end{align}
where $i=\sqrt{-1}$ is the complex unit, $u(x, y, t)$ is a complex-valued function of the time
variable $t$ and space variable $x, y$, $u_0(x,y)$ is the complex-valued function satisfying certain regularity,
and  $1<\alpha\leqslant 2$. The 2D Riesz fractional derivative $L_{\alpha}u\ (\alpha \in (1,2])$ is defined  in $\mathbb{R}^2$  as
\begin{equation*}
  L_{\alpha}u(x,y,t):=L^{\alpha}_xu(x,y,t)+L^{\alpha}_yu(x,y,t).
\end{equation*}
The fractional differential operator $L_x^{\alpha}u$ is
defined in $\mathbb{R}^2$ as follows
\begin{align}\label{1Dfr}
L^{\alpha}_xu(x,y,t):=\displaystyle-\frac{1}{2\cos(\frac{\alpha\pi}{2})}\left(_{-\infty}D^{\alpha}_xu(x,y,t)+_{x}D^{\alpha}_{\infty}u(x,y,t)\right),
\end{align}
in which the left and right Riemann-Liouville fractional derivatives are defined in \eqref{1Dfr} respectively,
\begin{align*}
_{-\infty}D^{\alpha}_xu(x,y,t)=\frac{1}{\Gamma(2-\alpha)}\frac{\partial^2}{\partial x^2}\int^x_{-\infty}\frac{u(\xi,y,t)}{(x-\xi)^{\alpha-1}}d\xi,\\
_{x}D^{\alpha}_{\infty}u(x,y,t)=\frac{1}{\Gamma(2-\alpha)}\frac{\partial^2}{\partial x^2}\int^{\infty}_x\frac{u(\xi,y,t)}{(\xi-x)^{\alpha-1}}d\xi.
\end{align*}
$L^{\alpha}_yu(x,y,t)$ can be defined similarly.
When $\alpha=2$, the 2D Riesz fractional derivative operator $L_{\alpha}$  reduces into the
standard Laplacian operator and the system reduces into the classical 2D nonlinear Schr\"{o}dinger's system \cite{He2017,Pan2020a}.

For problems \eqref{IVP1}--\eqref{IVP2}, by using the similar method in \cite{GHX2008},
we can obtain the following mass and energy conservation
\begin{equation}\label{QE1}
 Q(t) = Q(0), \qquad E(t) = E(0),
\end{equation}
where
\begin{equation*}
  Q(t) = \|u(\cdot,\cdot,t)\|_2,\qquad E(t) = \|(-\Delta)^{\frac{\alpha}{4}}u(\cdot,\cdot,t)\|_2^2-\frac{1}{2}\|u(\cdot,\cdot,t)\|_4^4.
\end{equation*}
$\|\cdot\|_2$ denotes the $L^2$-norm and $\|\cdot\|_4$ denotes the $L^{4}$-norm.
The detailed interpretation is provided in ~\ref{appendixa}.

In this paper,  we develop and analyze a linearized three-level implicit finite difference
scheme for solving the 2D SFNSE. The main contribution the paper is that the optimal pointwise
error estimate is obtained based on the discrete fractional Sobolev embedding theorem
in $H^{\alpha}$ ($1<\alpha\leq 2$) norm. And the technique for dealing with the nonlinear term in the error estimate process is completely new.  As far as we know, this is the first result  to show
the uniform convergence of the numerical solution for 2D SFNSE.

The rest of this paper is organized as follows. Section \ref{section2} summarizes the fractional Sobolev space and provides the 2D discrete fractional Sobolev inequality. The linearized implicit finite difference method is established in Section~\ref{section3}. Section~\ref{section4} analyzes the theoretical results including the unique solvability, convergence and stability. The numerical results are presented in Section \ref{section5}, which confirm our theoretical results. Finally, concluding remarks are provided in the last section.
\medskip

\section{Preliminaries}\label{section2}
\subsection{Spatial discretization}
In the past few years, numerous researchers have focused on the approximation of the Riesz fractional derivative, including Meerschaert \emph{et al.} \cite{Meerschaert1,Meerschaert2}, Yang \emph{et al.}~\cite{Yang}, Deng \emph{et al.}~\cite{Tian,Zhou}, Ortigueira~\cite{Ortigueira}, Celik and Duman~\cite{Celik}, Zhao~\emph{at al.} \cite{ZS2014} and many others. In current paper, Riesz fractional centered difference discretization for the fractional derivative will be adopted.
 \begin{lemma}\label{Lemma1}
  (see~\cite{Celik}){Let $f\in C^{5}(\mathbb{R})\cap L^1(\mathbb{R})$ and all spatial derivatives of $f$ up to fifth order are $L^1(\mathbb{R})$. Then}
\begin{equation*}
L^{\alpha}_xf(x) =-\frac{1}{h^{\alpha}_x}\sum^{\infty}_{j=-\infty}c^{\alpha}_{j}f(x-jh_x) + O(h^2),
\end{equation*}
for $1<\alpha \leqslant 2$, where
\begin{align}
c^{\alpha}_0=\frac{\Gamma(\alpha+1)}{\left(\Gamma(\frac{\alpha}{2}+1)\right)^2}, \quad c^{\alpha}_s=\frac{(-1)^{s}\Gamma(\alpha+1)}{\Gamma(\frac{\alpha}{2}-s+1)\Gamma(\frac{\alpha}{2}+s+1)}=\Big(1-\frac{\alpha+1}{\frac{\alpha}{2}+s}\Big)c_{s-1}^\alpha, \quad \textrm{for } s\in \mathbb{Z},\notag
\end{align}
 and
\begin{align}
c^{\alpha}_j=c^{\alpha}_{-j}\leqslant 0,\quad \textrm{for}\ j=\pm1, \pm 2, \cdots. \label{calpha}
\end{align}

 \end{lemma}
\subsection{Fractional Sobolev norm}

{An infinite 2-D mesh grid can be denoted by} $Z_h=\{(x_j=jh_x, y_k=kh_y), j, k\in Z\}$. For any two grid functions {$U=\{U_{j,k}\}, V=\{V_{j,k}\}$} defined on $Z_h$, the discrete inner product and the norms are defined as
\begin{align}\label{notation2}
 {(U,V)=h_x h_y\sum_{j,k\in Z}U_{j,k}{V}^*_{j,k},\quad \|U\|_2=\sqrt{(U,U)}, \quad \|U\|_{\infty}=\max_{j,k\in Z}{|U_{j,k}|},}
\end{align}
where $U^*$ is the complex conjugate of $U$. Let {$L^2_h=\left\{V|V\in Z_h,\|V\|_2<+\infty\right\}$}.

The semi-discrete Fourier transformation
$\hat{V}(k_1,k_2)\in L^2\left(\left[-\frac{\pi}{h_1},\frac{\pi}{h_1}\right]\right.$
$\left. \times\left[-\frac{\pi}{h_2},\frac{\pi}{h_2}\right]\right)$ is defined as
\begin{align}\label{Fourier}
\hat{V}(k_1,k_2)=\frac{h_xh_y}{2\pi}\sum_{j\in Z}\sum_{k\in Z}V_{j,k}e^{-\mathbf{i}k_1x_j-\mathbf{i}k_2y_k}
\end{align}
and the inversion formula is given by
\begin{align}\label{iFourier}
V_{j,k}=\frac{1}{2\pi}\int^{\frac{\pi}{h_y}}_{-\frac{\pi}{h_y}}\int^{\frac{\pi}{h_x}}_{-\frac{\pi}{h_x}}\hat{V}(k_1,k_2)e^{\mathbf{i}k_1x_j+\mathbf{i}k_2y_k}dk_1dk_2.
\end{align}
{The Parseval's identity naturally leads to that}
\begin{align}\label{parseval}
(U,V)=\int^{\frac{\pi}{h_y}}_{-\frac{\pi}{h_y}}\int^{\frac{\pi}{h_x}}_{-\frac{\pi}{h_x}}\hat{U}(k_1,k_2){\hat{V}^*(k_1,k_2)}dk_1dk_2.
\end{align}
Therefore, we can also use a definition $ \|U\|_2=\left(\int^{\frac{\pi}{h_y}}_{-\frac{\pi}{h_y}}\int^{\frac{\pi}{h_x}}_{-\frac{\pi}{h_x}}|\hat{U}(k_1,k_2)|^2dk_1dk_2\right)^{\frac{1}{2}}$.

For a given constant $\alpha \in(1,2]$, the fractional Sobolev norm $\|V\|_{H^{\alpha}}$  and semi-norms $|V|_{H^{\alpha}}$, $|V|_{H^{\frac{\alpha}{2}}}$ can be defined as follows
\begin{align}\label{Sobolev}
\| V\|^2_{H^{\alpha}}&=\int^{\pi/h_y}_{-\pi/h_y}\int^{\pi/h_x}_{-\pi/h_x}\left(1+|k_1|^{\alpha}+|k_2|^{\alpha}+2|k_1k_2|^{\alpha}+|k_1|^{2\alpha}+|k_2|^{2\alpha}\right)|\hat{V}(k_1,k_2)|^2dk_1dk_2,\\
|V|^2_{H^{\frac{\alpha}{2}}}&=\int^{\pi/h_y}_{-\pi/h_y}\int^{\pi/h_x}_{-\pi/h_x}\left(|k_1|^{\alpha}+|k_2|^{\alpha}\right)|\hat{V}(k_1,k_2)|^2dk_1dk_2,\\
|V|^2_{H^{\alpha}}&=\int^{\pi/h_y}_{-\pi/h_y}\int^{\pi/h_x}_{-\pi/h_x}\left(2|k_1k_2|^{\alpha}+|k_1|^{2\alpha}+|k_2|^{2\alpha}\right)|\hat{V}(k_1,k_2)|^2dk_1dk_2.
\end{align}
{Using these definitions, we have a discrete version of the summation property of Sobolev's norms}, $$\|V\|^2_{H^{\alpha}}=\|V\|^2+|V|^2_{H^{\frac{\alpha}{2}}}+|V|^2_{H^{\alpha}}.$$

Let {$H_h^\alpha:=\left\{V|V\in Z_h, \|V\|_{H^{\alpha}}<+\infty\right\}$.}

 \begin{lemma}\label{Lemma2}
 (Discrete fractional Sobolev's inequality) For any $1<\alpha\leqslant 2$ and $U\in H_h^{\alpha}$, it holds that
 \begin{align}
\|U\|_{\infty}\leqslant C_{\alpha}\|U\|_{H^{\alpha}},
\end{align}
where $C_{\alpha}$ is a positive constant that depends on $\alpha$ but does not depend on $h_x$ or $h_y$.
\end{lemma}
\begin{proof}
From the inversion of the Fourier transformation (\ref{iFourier}), we have
\begin{align}
\|U\|_{\infty}=&\;\max_{\substack{j, k\in Z}} \left|\frac{1}{2\pi}\int^{\frac{\pi}{h_y}}_{-\frac{\pi}{h_y}}\int^{\frac{\pi}{h_x}}_{-\frac{\pi}{h_x}}\hat{U}(k_1,k_2)e^{\mathbf{i}k_1x_j+\mathbf{i}k_2y_k}dk_1dk_2\right|\nonumber\\
\leqslant&\;\frac{1}{2\pi}\int^{\frac{\pi}{h_y}}_{-\frac{\pi}{h_y}}\int^{\frac{\pi}{h_x}}_{-\frac{\pi}{h_x}}\left|\hat{U}(k_1,k_2)\right|dk_1dk_2\nonumber\\
\leqslant&\;\frac{1}{2\pi}\int^{\frac{\pi}{h_y}}_{-\frac{\pi}{h_y}}\int^{\frac{\pi}{h_x}}_{-\frac{\pi}{h_x}}\left(1+|k_1|^{\alpha}\right)^{\frac{1}{2}}\left(1+|k_2|^{\alpha}\right)^{\frac{1}{2}}\left|\hat{U}(k_1,k_2)\right|\left(1+|k_1|^{\alpha}\right)^{-\frac{1}{2}}\left(1+|k_2|^{\alpha}\right)^{-\frac{1}{2}}dk_1dk_2\nonumber\\
\leqslant&\;\frac{1}{2\pi}\left(\int^{\frac{\pi}{h_y}}_{-\frac{\pi}{h_y}}
\int^{\frac{\pi}{h_x}}_{-\frac{\pi}{h_x}}\frac{1}{\left(1+|k_1|^{\alpha}\right)
\left(1+|k_2|^{\alpha}\right)}dk_1dk_2\right)^{\frac{1}{2}}\nonumber\\
&\;\times \left(\int^{\frac{\pi}{h_y}}_{-\frac{\pi}{h_y}}
\int^{\frac{\pi}{h_x}}_{-\frac{\pi}{h_x}}\left(1+|k_1|^{\alpha}\right)\left(1+|k_2|^{\alpha}\right)\left|\hat{U}(k_1,k_2)\right|^2dk_1dk_2\right)^{\frac{1}{2}}\nonumber\\
\leqslant&\;C_{\alpha}\|U\|_{H^{\alpha}},\notag
\end{align}
where $C_{\alpha}=\frac{1}{2\pi}\left(\int^{\frac{\pi}{h_y}}_{-\frac{\pi}{h_y}}\int^{\frac{\pi}{h_x}}_{-\frac{\pi}{h_x}}\frac{1}{\left(1+|k_1|^{\alpha}\right)
\left(1+|k_2|^{\alpha}\right)}dk_1dk_2\right)^{\frac{1}{2}}>0$.
\end{proof}

 \begin{lemma}\label{Lemma3}
 (Fractional seminorm equivalence) For every $1<\alpha\leqslant 2$ and $U\in H_h^{\alpha}$, the second-order difference scheme for the space fractional  derivatives is given by~\cite{Celik,Ortigueira}
 \begin{align}\label{difference}
 \mathcal{L}^{\alpha}_xU_{j,k}\triangleq-\frac{1}{h^{\alpha}_x}\sum^{+\infty}_{l=-\infty}c^{\alpha}_lU_{j-l,k},\quad \mathcal{L}^{\alpha}_yU_{j,k}\triangleq-\frac{1}{h^{\alpha}_y}\sum^{+\infty}_{m=-\infty}c^{\alpha}_mU_{j,k-m},
 \end{align}
then
 \begin{align}
\left(\frac{2}{\pi}\right)^{\alpha}|U|^2_{H^{\frac{\alpha}{2}}}\leqslant \left(-(\mathcal{L}^{\alpha}_x+\mathcal{L}^{\alpha}_y)U,U\right)\leqslant |U|^2_{H^{\frac{\alpha}{2}}}, \quad \forall\ U\in {H_h^\alpha},\notag
\end{align}
and
 \begin{align}
\left(\frac{2}{\pi}\right)^{2\alpha}|U|^2_{H^{\alpha}}\leqslant \left((\mathcal{L}^{\alpha}_x+\mathcal{L}^{\alpha}_y)U,(\mathcal{L}^{\alpha}_x+\mathcal{L}^{\alpha}_y)U\right)\leqslant |U|^2_{H^{\alpha}}, \quad \forall\ U\in {H_h^\alpha}.\notag
\end{align}
\end{lemma}
\begin{proof}
{Applying the semi-discrete Fourier transform on $\mathcal{L}^{\alpha}_xU_{j,k}+\mathcal{L}^{\alpha}_yU_{j,k}$, and using the generating functions}
 \[\sum^{+\infty}_{l=-\infty}c^{\alpha}_le^{\mathbf{i}h_xlk_1}=\left|2\sin\left(\frac{k_1h_x}{2}\right)\right|^{\alpha},\quad \sum^{+\infty}_{m=-\infty}c^{\alpha}_me^{\mathbf{i}h_ymk_2}=\left|2\sin\left(\frac{k_2h_y}{2}\right)\right|^{\alpha},\]
{we conclude that}
 \begin{align}
\widehat{ (\mathcal{L}^{\alpha}_x+\mathcal{L}^{\alpha}_y)U_{j,k}}=&-\left(\frac{1}{h^{\alpha}_x}\sum^{+\infty}_{l=-\infty}c^{\alpha}_le^{\mathbf{i}h_xlk_1}
+\frac{1}{h^{\alpha}_y}\sum^{+\infty}_{m=-\infty}c^{\alpha}_me^{\mathbf{i}h_ymk_2}\right)\hat{U}(k_1,k_2)\nonumber\\
=&-\left(\frac{1}{h^{\alpha}_x}\left|2\sin\left(\frac{k_1h_x}{2}\right)\right|^{\alpha}+\frac{1}{h^{\alpha}_y}
\left|2\sin\left(\frac{k_2h_y}{2}\right)\right|^{\alpha}\right)\hat{U}(k_1,k_2)\notag,
 \end{align}

The Parseval's identity leads to that
 \begin{align}
&\;\left(-(\mathcal{L}^{\alpha}_x+\mathcal{L}^{\alpha}_y)U,U\right)\nonumber\\
=&\;\int^{\frac{\pi}{h_y}}_{-\frac{\pi}{h_y}}\int^{\frac{\pi}{h_x}}_{-\frac{\pi}{h_x}}
\left(\frac{1}{h^{\alpha}_x}\left|2\sin\left(\frac{k_1h_x}{2}\right)\right|^{\alpha}+\frac{1}{h^{\alpha}_y}
\left|2\sin\left(\frac{k_2h_y}{2}\right)\right|^{\alpha}\right)|\hat{U}(k_1,k_2)|^2dk_1dk_2.\notag
 \end{align}
 Since $k_1\in [-\frac{\pi}{h_x},\frac{\pi}{h_x}]$ and $k_2\in [-\frac{\pi}{h_y},\frac{\pi}{h_y}]$, then $\frac{k_1h_x}{2}\in [-\frac{\pi}{2},\frac{\pi}{2}]$ and $\frac{k_2h_y}{2}\in [-\frac{\pi}{2},\frac{\pi}{2}]$. It hold a few estimates:
 $$\frac{2}{\pi}\left|\frac{k_1h_x}{2}\right|\leqslant \sin\left(\frac{k_1h_x}{2}\right)\leqslant \left|\frac{k_1h_x}{2}\right|,\quad \frac{2}{\pi}\left|\frac{k_2h_y}{2}\right|\leqslant \sin\left(\frac{k_2h_y}{2}\right)\leqslant \left|\frac{k_2h_y}{2}\right|,$$
 $$\left(\frac{2}{\pi}\right)^{\alpha}|k_1|^{\alpha}\leqslant\frac{1}{h^{\alpha}_x}\left|2\sin\left(\frac{k_1h_x}{2}\right)\right|^{\alpha}\leqslant |k_1|^{\alpha}, \quad \left(\frac{2}{\pi}\right)^{\alpha}|k_2|^{\alpha}\leqslant\frac{1}{h^{\alpha}_y}\left|2\sin\left(\frac{k_2h_y}{2}\right)\right|^{\alpha}\leqslant |k_2|^{\alpha}.$$

 Therefore,
  \begin{align}
& \left(\frac{2}{\pi}\right)^{\alpha}\int^{\frac{\pi}{h_y}}_{-\frac{\pi}{h_y}}\int^{\frac{\pi}{h_x}}_{-\frac{\pi}{h_x}}\left(|k_1|^{\alpha}+|k_2|^{\alpha}\right)
|\hat{U}(k_1,k_2)|^2dk_1dk_2 \notag\\
\leqslant &\;\left(-(\mathcal{L}^{\alpha}_x+\mathcal{L}^{\alpha}_y)U,U\right)\notag \\
\leqslant&\; \int^{\frac{\pi}{h_y}}_{-\frac{\pi}{h_y}}\int^{\frac{\pi}{h_x}}_{-\frac{\pi}{h_x}}\left(|k_1|^{\alpha}+|k_2|^{\alpha}\right)
|\hat{U}(k_1,k_2)|^2dk_1dk_2,\notag
 \end{align}
{which can be understood as}
 \begin{align}
\left(\frac{2}{\pi}\right)^{\alpha}|U|^2_{H^{\frac{\alpha}{2}}}\leqslant \left(-(\mathcal{L}^{\alpha}_x+\mathcal{L}^{\alpha}_y)U,U\right)\leqslant |U|^2_{H^{\frac{\alpha}{2}}}, \quad \forall\ U\in {H_h^\alpha}.\notag
\end{align}

 Again, using the Parseval's identity, we have
 \begin{align}
&\left((\mathcal{L}^{\alpha}_x+\mathcal{L}^{\alpha}_y)U,(\mathcal{L}^{\alpha}_x+\mathcal{L}^{\alpha}_y)U\right)\notag\\
=&\int^{\frac{\pi}{h_y}}_{-\frac{\pi}{h_y}}\int^{\frac{\pi}{h_x}}_{-\frac{\pi}{h_x}}\left(\frac{1}{h^{\alpha}_x}
\left|2\sin\left(\frac{k_1h_x}{2}\right)\right|^{\alpha}+\frac{1}{h^{\alpha}_y}\left|2\sin\left(\frac{k_2h_y}{2}\right)\right|^{\alpha}\right)^2|\hat{U}(k_1,k_2)|^2dk_1dk_2.\notag
 \end{align}

 Therefore,
  \begin{align}
&\left(\frac{2}{\pi}\right)^{2\alpha}\int^{\frac{\pi}{h_y}}_{-\frac{\pi}{h_y}}\int^{\frac{\pi}{h_x}}_{-\frac{\pi}{h_x}}\left(|k_1|^{\alpha}+|k_2|^{\alpha}\right)^2
|\hat{U}(k_1,k_2)|^2dk_1dk_2\notag \\
\leqslant&\; \left((\mathcal{L}^{\alpha}_x+\mathcal{L}^{\alpha}_y)U,(\mathcal{L}^{\alpha}_x+\mathcal{L}^{\alpha}_y)U\right)\notag \\
\leqslant&\; \int^{\frac{\pi}{h_y}}_{-\frac{\pi}{h_y}}\int^{\frac{\pi}{h_x}}_{-\frac{\pi}{h_x}}\left(|k_1|^{\alpha}+|k_2|^{\alpha}\right)^2
|\hat{U}(k_1,k_2)|^2dk_1dk_2.\notag
 \end{align}
i.e.,
 \begin{align}
\left(\frac{2}{\pi}\right)^{2\alpha}|U|^2_{H^{\alpha}}\leqslant \left((\mathcal{L}^{\alpha}_x+\mathcal{L}^{\alpha}_y)U,(\mathcal{L}^{\alpha}_x+\mathcal{L}^{\alpha}_y)U\right)\leqslant |U|^2_{H^{\alpha}}, \quad \forall\ U\in {H_h^\alpha}.\notag
\end{align}
\end{proof}

\section{A three-level linearized implicit difference scheme}\label{section3}

Since the solution of problem (\ref{IVP1})-(\ref{IVP2}) {vanishes} as $(x,y)\rightarrow\infty$, {then in practical calculations, the original problem can be regarded as that defined on} a sufficiently large but bounded domain with homogeneous Dirichlet boundary condition. Namely,
\begin{align}\label{BC}
u(x,y,t)=0,\quad (x,y)\ {\rm on}\ \partial\Omega,\ t\in [0,T],
\end{align}
where $\Omega=[a,b]\times[c,d]$,  $a,c$ are sufficiently large negative numbers and  $b,d$ are sufficiently large positive numbers such that the truncation error is negligible.
In what follows, we will make no distinction between $u(x,y,t)$ defined in $\mathbb{R}^2$ and $u(x,y,t)$ defined in $\Omega=[a,b]\times[c,d]$, {and always suppose that the equation is augmented by the zero-boundary constraint}.

{Let $N$, $M_x$ and $M_y$ be positive integers, and define $\Delta t= T/N$, $h_x={(b-a)}/{M_x}$, $h_y={(d-c)}/{M_y}$ to be the time step and spatial mesh-size, respectively}. {Now
$\Omega\times [0,T]$ can therefore be covered by}
\[\Omega^0_h=\{(x_i,y_j,t_n)|x_j=jh_x,\;y_k=kh_y,\;t_n=n\Delta t, \;j=0,\cdots,M_x, \;k=0,\cdots, M_y, \;n=0,\cdots, N\}.\]
In addition, let  $U^n=(U^n_{j,k})_{(M_x+1)(M_y+1)}$ be the numerical solution at time level $t=t_n$.
{Following this definition, we have the discrete Dirichlet boundary condition} (\ref{BC})
\begin{equation}\label{Boundary}
  U^{n}_{0,k}=U^n_{M_x,k}=U^{n}_{j,0}=U^{n}_{j,M_y}=0,\quad n=0,\cdots,N.
\end{equation}

 For  {every} grid function $U^n$, we define four difference operators:
 \begin{align*}
 &U^{\bar{n}}_{j,k}=\frac{U^{n+1}_{j,k}+U^{n-1}_{j,k}}{2}, \quad \delta_{t}U^{n}_{j,k}=\frac{U^{n+1}_{j,k}-U^{n-1}_{j,k}}{2\tau}, \\
  &\delta_t U_{j,k}^{n+\frac{1}{2}} = \frac{1}{\tau}(U^{n+1}_{j,k}-U^{n}_{j,k}),  \quad U^{\bar{\bar{n}}}_{j,k}=\frac{U^{n+2}_{j,k}+2U^{n}_{j,k}+U^{n-2}_{j,k}}{4}.
 \end{align*}
We also need {the discrete} inner product and norms for each two grid functions $U, V$. Define
\begin{align}\label{notation1}
(U,V)=h_x h_y\sum^{M_x-1}_{j=1}\sum^{M_y-1}_{k=1}U_{j,k}{V}^*_{j,k},\quad \|U\|_2=\sqrt{(U,U)},  \quad \|U\|_{\infty}=\max\limits_{\substack{1\leqslant j\leqslant M_x-1\\ 1\leqslant k\leqslant M_y-1}}|U^n_{j,k}|.
\end{align}

{We should notice that for any $U\in \Omega^{0}_h$, we can extend the grid function into an infinite-grid-defined function by letting} $U_{j,k}=0$ when $(x_j=j\Delta x,y_k=k\Delta y)$ is {outside} the domain $\Omega$, Thus the norms and inner product defined in (\ref{notation2}) can also be defined for grid functions in $\Omega^{0}_h$. Moreover, Lemma~\ref{Lemma1} and Lemma~\ref{Lemma2}  are also valid for every $U\in \Omega^{0}_h$. {As a result, we will not make distinctions} between the inner products and norms defined in (\ref{notation2}) and (\ref{notation1}) {in all the later sections}.

Note that Eq. \eqref{IVP1}  can be rewritten as
\begin{align}\label{equation2}
u_t=\mathbf{i}L^{\alpha}_xu+\mathbf{i}L^{\alpha}_yu+\mathbf{i}|u|^2u.
\end{align}
For time discretization, we use a three-level {linear} implicit
scheme for (\ref{IVP1}) around $t=t_n$, which causes a second-order error of time step:
\begin{align}\label{equation3}
\frac{u^{n+1}-u^{n-1}}{2\tau}=\mathbf{i}L^{\alpha}_x\frac{u^{n+1}+u^{n-1}}{2}+\mathbf{i}L^{\alpha}_y\frac{u^{n+1}+u^{n-1}}{2}+\mathbf{i}|u^n|^2\frac{u^{n+1}+u^{n-1}}{2}+O(\tau^2).
\end{align}
After collecting the terms of $u^{n+1}$ and $u^{n-1}$ in (\ref{equation3}), the above relation is changed into
\begin{align}\label{equation4}
\left(\frac{1}{\tau}-\mathbf{i}|u^n|^2-\mathbf{i}L^{\alpha}_x-\mathbf{i}L^{\alpha}_y\right)u^{n+1}
=\left(\frac{1}{\tau}+\mathbf{i}|u^n|^2+\mathbf{i}L^{\alpha}_x+\mathbf{i}L^{\alpha}_y\right)u^{n-1}+O(\tau^2).
\end{align}

Under the assumption of the homogenous boundary condition (\ref{BC}), the second-order difference scheme for  space fractional derivatives (\ref{difference}) {are given below}
\begin{align} (\mathcal{L}^{\alpha}_xU)_{j,k}=-\frac{1}{h^{\alpha}_x}\sum^{j-1}_{s=j-M_x+1}c^{\alpha}_{s}U_{j-s,k}=-\frac{1}{h^{\alpha}_x}\sum^{M_x-1}_{s=1}c^{\alpha}_{j-s}U_{s,k},\\ (\mathcal{L}^{\alpha}_yU)_{j,k}=-\frac{1}{h^{\alpha}_y}\sum^{k-1}_{s=k-M_y+1}c^{\alpha}_{s}U_{j,k-s}=-\frac{1}{h^{\alpha}_y}\sum^{M_y-1}_{s=1}c^{\alpha}_{k-s}U_{j,s}.
\end{align}

{At last, our linear finite difference scheme is summarized as follows:}
\begin{align}\label{scheme1}
\left(\frac{1}{\tau}-\mathbf{i}|U^n_{j,k}|^2-\mathbf{i}\mathcal{L}^{\alpha}_x-\mathbf{i}\mathcal{L}^{\alpha}_y\right)U^{n+1}_{j,k}=\left(\frac{1}{\tau}+\mathbf{i}|U^n_{j,k}|^2+\mathbf{i}\mathcal{L}^{\alpha}_x+\mathbf{i}\mathcal{L}^{\alpha}_y\right)U^{n-1}_{j,k},
\end{align}
for $1\leqslant j\leqslant M_x-1, 1\leqslant k\leqslant M_y-1, 1\leqslant n\leqslant N-1$. {For such a time-iterating scheme, an initialization is always needed for starting the loop. Considering the given initial data,  the solution at the initial step is then given by}
\begin{equation}\label{scheme2}
U^{0}_{j,k}=(u_0)_{j,k}=u_0(x_j,y_k) \quad 0<j<M_x, \quad 0<k<M_y.
\end{equation}

Moreover, Since the linear scheme (\ref{scheme1}) involves three time levels, {the first step values $\lbrace U^1_{j,k} \rbrace$ are also indispensable to begin the loop. We accomplish this by using Taylor's expansion:}
\begin{align}
u^1=&\;u_0+\tau u_t(x,0)+\frac{\tau^2}{2}u_{tt}(x,\bar{\tau})\nonumber\\
=&\;u_0+\tau\left(\mathbf{i}L^{\alpha}_xu_0+\mathbf{i}L^{\alpha}_yu_0+\mathbf{i}|u_0|^2u_0\right)+\frac{\tau^2}{2}u_{tt}(x,\bar{\tau}),
\end{align}
for some $\bar{\tau}\in (0,\tau)$, and equation (\ref{equation2}) {is used here to substitute the time derivative.} Hence in the numerical {approximation}, $\lbrace U^1_{j,k} \rbrace$ can be  obtained by  the following second-order scheme
\begin{align}\label{scheme3}
U^1_{j,k}&=U^0_{j,k}+\tau\left(\mathbf{i}(\mathcal{L}^{\alpha}_xu_0)_{j,k}
+\mathbf{i}(\mathcal{L}^{\alpha}_yu_0)_{j,k}+\mathbf{i}|(u_0)_{j,k}|^2(u_0)_{j,k}\right),
\end{align}
for $1\leqslant j\leqslant M_x-1, 1\leqslant k\leqslant M_y-1$.

\section{Theoretical analysis}\label{section4}
The following conclusions are essential for the theoretical analysis of the numerical solution.
\begin{lemma}\label{Lemmadelta}
For any $U,V\in \Omega^{0}_h$, {the following product can be taken apart}:
\begin{align}
\delta_t \left( U^n V^n \right) = U^{\bar{n}}\delta_t V^n + V^{\bar{n}}\delta_t U^n.
\end{align}
The multiplication in the above formula and the proof below is an element-wise operation of grid functions in $\Omega^{0}_h$.
\end{lemma}
\begin{proof}
Notice that
\begin{align*}
\delta_t\left( U^n V^n \right) =&\; \dfrac{U^{n+1}V^{n+1} - U^{n-1}V^{n-1}}{2\tau} \notag\\
                               =&\; \dfrac{\left( U^{n+1}-U^{n-1} \right)V^{n+1} + U^{n-1}\left( V^{n+1}-V^{n-1} \right)}{2\tau} \notag\\
                               =&\; \delta_t U^n V^{n+1} + U^{n-1}\delta_t V^n.
\end{align*}
On the other hand, {we can compute symmetrically}
\begin{align*}
\delta_t\left( U^n V^n \right) = &\;\dfrac{U^{n+1}V^{n+1} - U^{n-1}V^{n-1}}{2\tau}\notag\\
                               = &\;\dfrac{\left( V^{n+1}-V^{n-1} \right)U^{n+1} + V^{n-1}\left( U^{n+1}-U^{n-1} \right)}{2\tau} \notag\\
                               = &\;\delta_t V^n U^{n+1} + V^{n-1}\delta_t U^n.
\end{align*}
Taking the average of the two equations, we conclude that
\begin{align*}
\delta_t\left( U^n V^n \right) = \dfrac{U^{n+1}+U^{n-1}}{2}\delta_t V^n  + \dfrac{V^{n+1}+V^{n-1}}{2}\delta_t U^n = U^{\bar{n}}\delta_t V^n + V^{\bar{n}}\delta_t U^n.
\end{align*}
\end{proof}

 \begin{lemma}\label{Lemma4}
For any $ U,V\in \Omega^{0}_h$,  there exists a linear operator $\Lambda^{\alpha}$ such that
 \begin{align}\label{lemma4eq}
\left(-(\mathcal{L}^{\alpha}_x+\mathcal{L}^{\alpha}_y)U,V\right)=(\Lambda^{\alpha}U,\Lambda^{\alpha}V).
\end{align}
\end{lemma}
\begin{proof}
 Let
\begin{align}
\mathbf{C}_x&=\frac{1}{h^{\alpha}_x}\begin{pmatrix}
c^{\alpha}_0& c^{\alpha}_{-1}&\cdots&c^{\alpha}_{-M_x+2}\\
c^{\alpha}_1& c^{\alpha}_0&\cdots&c^{\alpha}_{-M_x+3}\\
\vdots& \vdots&\ddots&\vdots\\
c^{\alpha}_{M_x-2}& c^{\alpha}_{M_x-3}&\cdots&c^{\alpha}_0\\
\end{pmatrix}, \quad \mathbf{C}_y=\frac{1}{h^{\alpha}_y}\begin{pmatrix}
c^{\alpha}_0& c^{\alpha}_{-1}&\cdots&c^{\alpha}_{-M_y+2}\\
c^{\alpha}_1& c^{\alpha}_0&\cdots&c^{\alpha}_{-M_y+3}\\
\vdots& \vdots&\ddots&\vdots\\
c^{\alpha}_{M_y-2}& c^{\alpha}_{M_y-3}&\cdots&c^{\alpha}_0\\
\end{pmatrix},\notag
\end{align}
$$\mathbf{D}=\mathbf{I}_{M_y-1}\otimes \mathbf{C}_{x}+\mathbf{C}_{y}\otimes \mathbf{I}_{M_x-1},$$
$$\mathcal{U}=(U_{1,1},\cdots,U_{M_x-1,1},U_{1,2},\cdots,U_{M_x-1,2},\cdots,U_{1,M_y-1},\cdots,U_{M_x-1,M_y-1})^T,$$
$$\mathcal{V}=(V_{1,1},\cdots,V_{M_x-1,1},V_{1,2},\cdots,V_{M_x-1,2},\cdots,V_{1,M_y-1},\cdots,V_{M_x-1,M_y-1})^T.$$
Since $\mathbf{C}_x$ and $\mathbf{C}_y$ are {both} real symmetric positive definite matrices~\cite{WH2015},
$\mathbf{D}$ is also a real symmetric positive definite matrix. Form the spectral theorem, there exist a real orthogonal
matrix $\mathbf{P}$ and a diagonal matrix $\mathbf{A}$ with positive diagonal entries such that
$$\mathbf{D}=\mathbf{PAP}^T=(\mathbf{PA}^{\frac{1}{2}}\mathbf{P}^T)(\mathbf{PA}^{\frac{1}{2}}\mathbf{P}^T)^T=\mathbf{L}^T\mathbf{L},$$
where $\mathbf{L}=\mathbf{PA}^{\frac{1}{2}}\mathbf{P}^T$ is also a real symmetric positive definite matrix.

Now that
 \begin{align}
-(\mathcal{L}^{\alpha}_x+\mathcal{L}^{\alpha}_y)U=\mathbf{D}\mathcal{U},
\end{align}
hence
 \begin{align}
\left(-(\mathcal{L}^{\alpha}_x+\mathcal{L}^{\alpha}_y)U,V\right)=h_xh_y\mathcal{V}^H \mathbf{D}\mathcal{U}=h_xh_y\mathcal{V}^H\mathbf{L}^T\mathbf{L}\mathcal{U}=h_xh_y(\mathbf{L}\mathcal{V})^H\mathbf{L}\mathcal{U}=(\mathbf{L}U,\mathbf{L}V),
\end{align}
where $\mathcal{V}^H$ is the Hermitian transpose of $\mathcal{V}$.
If we define the linear operator $\Lambda^{\alpha}U=\mathbf{L}U$ for any $ U\in \Omega^{0}_h$ , then we can obtain (\ref{lemma4eq}).
\end{proof}

By using Lemma \ref{Lemma4}, the following lemma is easy to verify.
 \begin{lemma}\label{Lemma5}
 For any $ U^n\in \Omega^{0}_h$, we have
 \begin{align}
&Im\left((\mathcal{L}^{\alpha}_x+\mathcal{L}^{\alpha}_y)U,U\right)=0,\\
&Re\left(-(\mathcal{L}^{\alpha}_x+\mathcal{L}^{\alpha}_y)U^{\bar{n}},\delta_{t} U^n\right)=\frac{1}{4\tau}\left(\parallel\Lambda^{\alpha} U^{n+1}\parallel_2^2-\parallel\Lambda^{\alpha} U^{n-1}\parallel_2^2\right).
\end{align}
\end{lemma}

{
 \begin{lemma}\label{Lemma6}
 (Discrete Gronwall's inequality~\cite{Hu2015,Holte}). Let $\{u_k\}$ and $\{w_k\}$ be nonnegative sequences and $\alpha$ {be} a nonnegative constant. They together satisfy
 \begin{equation*}
   u_n\leqslant \alpha + \sum_{0\leqslant k<n}w_k u_k \quad \textrm{for } n\geq 0.
 \end{equation*}
Then for all $n\geq 0$, it holds that
\begin{equation*}
   u_n\leqslant \alpha \exp\left(\sum_{0\leqslant k<n} w_k\right).
\end{equation*}
\end{lemma}}

\subsection{Energy conservation}\label{conservation}
It {can be verified} that \eqref{scheme1} on the mesh grids is equivalent to
\begin{equation}\label{scheme1_re}
 \mathbf{i} \frac{U^{n+1}_{j,k}-U^{n-1}_{j,k}}{2\tau}+ \left(\mathcal{L}^{\alpha}_x\frac{U_{j,k}^{n+1}+U_{j,k}^{n-1}}{2}+\mathcal{L}^{\alpha}_y\frac{U_{j,k}^{n+1}+U_{j,k}^{n-1}}{2}\right)+ |U_{j,k}^n|^2\frac{U_{j,k}^{n+1}+U_{j,k}^{n-1}}{2}=0.
\end{equation}

Now we {prove} the following discrete energy conservative law, which is similar to the continuous case.
\begin{theorem}\label{thm1}
  The schemes \eqref{scheme1_re}{ with the initialization} \eqref{scheme3} are conservative in the sense {of discrete energy}:
  \begin{align}
   E^n  = E^0,\quad 1 \leqslant n \leqslant N-1,\label{Conser1}
  \end{align}
  where
  \begin{equation}\label{Conser2}
    E^n = \frac{1}{2} (\|\Lambda^{\alpha}U^{n+1}\|_2^2+\|\Lambda^{\alpha}U^n\|_2^2) -\frac{1}{2}h_xh_y\sum_{j=1}^{M_x-1}\sum_{k=1}^{M_y-1}|U_{j,k}^n|^2|U_{j,k}^{n+1}|^2.
  \end{equation}
\end{theorem}
\begin{proof}
  Taking inner product of \eqref{scheme1_re} with $-\delta_tU^n$, we have
  \begin{equation}\label{Conser3}
    -\mathbf{i} \|\delta_tU^n\|_2^2 - (\mathcal{L}_x^{\alpha}U^{\bar{n}}+ \mathcal{L}_y^{\alpha} U^{\bar{n}}, \delta_t{U}^n) - (|U^n|^2U^{\bar{n}},\delta_t {U}^n)=0.
  \end{equation}
{Now consider the real part of }\eqref{Conser3}. {It follows from Lemma} \ref{Lemma4}{ and Lemma} \ref{Lemma5} {that}
  \begin{align}
    \frac{1}{4\tau} \left(\parallel\Lambda^{\alpha} U^{n+1}\parallel_2^2-\parallel\Lambda^{\alpha} U^{n-1}\parallel_2^2\right) - h_xh_y\sum_{j=1}^{M_x-1}\sum_{k=1}^{M_y-1}|U_{j,k}^n|^2\cdot \frac{1}{4\tau} (|U_{j,k}^{n+1}|^2-|U_{j,k}^{n-1}|^2)=0.\notag
  \end{align}
{Therefore},
\begin{align}
  &\frac{1}{2}(\|\Lambda^{\alpha}U^{n+1}\|_2^2+\|\Lambda^{\alpha}U^n\|_2^2) -\frac{1}{2} h_xh_y\sum_{j=1}^{M_x-1}\sum_{k=1}^{M_y-1}|U_{j,k}^n|^2\cdot |U_{j,k}^{n+1}|^2\notag\\
  =&\; \frac{1}{2}(\|\Lambda^{\alpha}U^{n}\|_2^2+\|\Lambda^{\alpha}U^{n-1}\|_2^2) -\frac{1}{2} h_xh_y\sum_{j=1}^{M_x-1}\sum_{k=1}^{M_y-1}|U_{j,k}^{n-1}|^2\cdot |U_{j,k}^{n}|^2.\notag
\end{align}
This completes the proof.
\end{proof}

\subsection{Solvability of difference scheme}

\begin{theorem}
 The finite difference scheme \eqref{Boundary}, (\ref{scheme1})--(\ref{scheme2}) and \eqref{scheme3} is uniquely solvable.
\end{theorem}

\begin{proof}
Since the schemes  \eqref{Boundary}, (\ref{scheme1})--(\ref{scheme2}), \eqref{scheme3} {are derived} form a linear system of equations,
it is sufficient to show that the homogeneous linear system has only zero solution.
The data of initial level and the first level have been uniquely determined by \eqref{scheme2} and \eqref{scheme3}.
Now we suppose that $U^1,\cdots,U^n$ have already been {uniquely} obtained, {then the iteration yields a new} linear system for $U^{n+1}$
\begin{align}
   \frac{\mathbf{i}}{\tau}U^{n+1}_{j,k}+ (\mathcal{L}^{\alpha}_x U_{j,k}^{n+1}+\mathcal{L}^{\alpha}_y U_{j,k}^{n+1})+ |U_{j,k}^n|^2 U_{j,k}^{n+1}=0.
\end{align}
Computing the inner product of $U^{n+1}$ with both sides of the above equation, and following the result from Lemma \ref{Lemma4}, we can show that $\|U^n\|_2=0$, $n=2,3,\cdots$.{ This implies, by a closed induction, }that
the finite difference schemes \eqref{Boundary}, (\ref{scheme1})--(\ref{scheme2}), \eqref{scheme3} are uniquely solvable.
\end{proof}

\subsection{$L^{\infty}$ convergence and stability}

Let $u(x,y,t)$ be the exact solution of problem (\ref{IVP1})-(\ref{IVP2}) and $U^n_{j,k}$ be the solution of the numerical schemes \eqref{Boundary}, (\ref{scheme1})-(\ref{scheme2}), \eqref{scheme3}. Define  $u^n_{j,k}=u(x_j,y_k,t_n)$. Then the error function is the difference between $u$ and $U$: $$e^n_{j,k}=u^n_{j,k}-U^n_{j,k},\quad j=1,2,\cdots, M_x-1, \;\; k=1,2,\cdots, M_y-1,\;\; n=1,2,\cdots,N.$$

By substituting $U$ with $u$ in the schemes \eqref{Boundary}, (\ref{scheme1})-(\ref{scheme2}), we can define the truncation errors as follows:
\begin{align}\label{trun1}
r^n_{j,k}=\frac{1}{2}\left(\frac{1}{\tau}-\mathbf{i}|u^n_{j,k}|^2-\mathbf{i}\mathcal{L}^{\alpha}_x-\mathbf{i}\mathcal{L}^{\alpha}_y\right)u^{n+1}_{j,k}-\frac{1}{2}\left(\frac{1}{\tau}+\mathbf{i}|u^n_{j,k}|^2+\mathbf{i}\mathcal{L}^{\alpha}_x+\mathbf{i}\mathcal{L}^{\alpha}_y\right)u^{n-1}_{j,k},
\end{align}
for $1\leqslant j\leqslant M_x-1, 1\leqslant k\leqslant M_y-1, 1\leqslant n \leqslant N-1$.

Through simple calculations, we see that
\begin{align}\label{trun2}
r^n_{j,k}  =\delta_tu^n_{j,k}-\mathbf{i}(\mathcal{L}^{\alpha}_x+\mathcal{L}^{\alpha}_y)u^{\bar{n}}_{j,k}-\mathbf{i}|u^n_{j,k}|^2u^{\bar{n}}_{j,k},
\end{align}
for $1\leqslant j\leqslant M_x-1, 1\leqslant k\leqslant M_y-1, 1\leqslant n \leqslant N-1$.

{However, it follows from} the difference scheme (\ref{scheme1}) that
\begin{align}\label{diff1}
0=\delta_tU^n_{j,k}-\mathbf{i}(\mathcal{L}^{\alpha}_x+\mathcal{L}^{\alpha}_y)U^{\bar{n}}_{j,k}-\mathbf{i}|U^n_{j,k}|^2U^{\bar{n}}_{j,k},
\end{align}
for $1\leqslant j\leqslant M_x-1, 1\leqslant k\leqslant M_y-1,  1\leqslant n \leqslant N-1$.

{This implies that we obtain a relation between the error and the truncation error, by} subtracting (\ref{trun2})  from (\ref{diff1}):
\begin{align}
r^n_{j,k}=&\delta_{t}e^n_{j,k}-\mathbf{i}(\mathcal{L}^{\alpha}_x+\mathcal{L}^{\alpha}_y)e^{\bar{n}}_{j,k}-\mathbf{i}P^n_{j,k},\label{truncation}
\end{align}
where \begin{align*}
P^n_{j,k}&=\left(|u^n_{j,k}|^2u^{\bar{n}}_{j,k}-|U^n_{j,k}|^2U^{\bar{n}}_{j,k}\right),
\end{align*}
for $1\leqslant j\leqslant M_x-1, 1\leqslant k\leqslant M_y-1,  1\leqslant n \leqslant N-1$.

 \begin{lemma}\label{Lemma7}
Suppose that the solution of problem (\ref{IVP1})-(\ref{IVP2}) is sufficiently smooth {and satisfies the homogeneous Dirichlet's boudary condition}. Then it holds that
\begin{align}
&|r^n_{j,k}|\leqslant C_R(\tau^2+h^2_x+h^2_y),\quad \left|\delta_t r^n_{j,k}\right|\leqslant C_R(\tau^2+h^2_x+h^2_y), \nonumber\\
 &\qquad 1\leqslant j\leqslant M_x-1, 1\leqslant k \leqslant M_y-1, 1\leqslant n \leqslant N-1,
\end{align}
where $C_R$ is a positive constant independent of $\tau$, $h_x$ and $h_y$.
\end{lemma}
\begin{proof}
 Using the Taylor's expansion of the exact solution $u(x,y,t)$ at $(x_j,y_k,t_n)$, we have
\begin{align}
&\frac{u(x_j,y_k,t_{n+1})-u(x_j,y_k,t_{n-1})}{2\tau}-\partial_tu(x_j,y_k,t_n)\notag\\
=&\;\frac{\tau^2}{4}\int^{1}_{0}\left(\frac{\partial^3 u(x_j,y_k, t_n+s\tau)}{\partial t^3}+\frac{\partial^3 u(x_j,y_k, t_n-s\tau)}{\partial t^3}\right)(1-s)^2ds.\notag
\end{align}
From the proof in~\cite{Celik}, we have
\begin{align}
&\mathcal{L}^{\alpha}_x\frac{u(x_j,y_k,t_{n+1})+u(x_j,y_k,t_{n-1})}{2}-L^{\alpha}_xu(x_j,y_k,t_n)\nonumber\\
=&\; \left(\mathcal{L}^{\alpha}_x\frac{u(x_j,y_k,t_{n+1})+u(x_j,y_k,t_{n-1})}{2}-L^{\alpha}_x\frac{u(x_j,y_k,t_{n+1})+u(x_j,y_k,t_{n-1})}{2}\right)\nonumber\\
 &\; +L^{\alpha}_x\left(\frac{u(x_j,y_k,t_{n+1})+u(x_j,y_k,t_{n-1})}{2}-u(x_j,y_k,t_n)\right)\nonumber\\
=&\; \frac{\tau^2}{2}L^{\alpha}_x\int^{1}_{0}\left(\frac{\partial^2 u(x_j,y_k, t_n+s\tau)}{\partial t^2}+\frac{\partial^2 u(x_j,y_k, t_n-s\tau)}{\partial t^2}\right)(1-s)ds+O(h^2_x).\notag
\end{align}
Similarly,
\begin{align}
&\mathcal{L}^{\alpha}_y\frac{u(x_j,y_k,t_{n+1})+u(x_j,y_k,t_{n-1})}{2}-L^{\alpha}_yu(x_j,y_k,t_n)\nonumber\\
=&\; \frac{\tau^2}{2}L^{\alpha}_y\int^{1}_{0}\left(\frac{\partial^2 u(x_j,y_k, t_n+s\tau)}{\partial t^2}+\frac{\partial^2 u(x_j,y_k, t_n-s\tau)}{\partial t^2}\right)(1-s)ds+O(h^2_y).\notag
\end{align}
In addition, the nonlinear term can also be expanded
\begin{align}
&|u(x_j,y_k,t_n)|^2\frac{u(x_j,y_k,t_{n+1})+u(x_j,y_k,t_{n-1})}{2}-|u(x_j,y_k,t_n)|^2u(x_j,y_k,t_n)\nonumber\\
=&\; |u(x_j,y_k,t_n)|^2\left(\frac{u(x_j,y_k,t_{n+1})+u(x_j,y_k,t_{n-1})}{2}-u(x_j,y_k,t_n)\right)\nonumber\\
=&\; |u(x_j,y_k,t_n)|^2\frac{\tau^2}{2}\int^{1}_{0}\left(\frac{\partial^2 u(x_j,y_k, t_n+s\tau)}{\partial t^2}+\frac{\partial^2 u(x_j,y_k, t_n-s\tau)}{\partial t^2}\right)(1-s)ds.\notag
\end{align}
{Summarizing all the computations, it holds}
\begin{align}
r^n_{j,k}=&\; \frac{\tau^2}{4}\int^{1}_{0}\left(\frac{\partial^3 u(x_j,y_k, t_n+s\tau)}{\partial t^3}+\frac{\partial^3 u(x_j,y_k, t_n-s\tau)}{\partial t^3}\right)(1-s)^2ds\nonumber\\
&\; -\mathbf{i}\frac{\tau^2}{2}(L^{\alpha}_x+L^{\alpha}_y)\int^{1}_{0}\left(\frac{\partial^2 u(x_j,y_k, t_n+s\tau)}{\partial t^2}+\frac{\partial^2 u(x_j,y_k, t_n-s\tau)}{\partial t^2}\right)(1-s)ds\nonumber\\
&\; -\frac{\tau^2}{2}\mathbf{i}|u(x_j,y_k,t_n)|^2\int^{1}_{0}\left(\frac{\partial^2 u(x_j,y_k, t_n+s\tau)}{\partial t^2}+\frac{\partial^2 u(x_j,y_k, t_n-s\tau)}{\partial t^2}\right)(1-s)ds\nonumber\\
&\; +O(h^2_x)+O(h^2_y).\label{trun11}
\end{align}
Therefore,
\begin{align}
r^n_{j,k}=O(\tau^2+h^2_x+h^2_y).\notag
\end{align}
By using (\ref{trun11}), it will be easy to see that
\begin{align}
\delta_t r^n_{j,k}=O(\tau^2+h^2_x+h^2_y),\quad  1\leqslant j\leqslant M_x-1,1\leqslant k \leqslant M_y-1, 1\leqslant n \leqslant N-1,\notag
\end{align}
completing the proof.
\end{proof}

Following a similar proof of Lemma 9 in \cite{He2018}, we can obtain the lemma below.
 \begin{lemma}\label{Lemma8}
 Suppose that the solution of problem (\ref{IVP1})-(\ref{IVP2}) is sufficiently smooth and {vanishes} as $(x,y)\rightarrow\infty$.
 Then for the difference scheme (\ref{scheme3}), we have
 \begin{equation}\label{e1}
\begin{aligned}
|e^1_{j,k}|&\leqslant C_e(\tau^2+\tau h^2_x+\tau h^2_y),\quad  |(\mathcal{L}^{\alpha}_x +\mathcal{L}^{\alpha}_y) e^1_{j,k}|\leqslant C_e(\tau^2+\tau h^{2}_x+\tau h^{2}_y), \notag\\
& 1\leqslant j\leqslant M_x-1,\quad 1\leqslant k \leqslant M_y-1, \quad 1\leqslant n \leqslant N-1,
\end{aligned}
 \end{equation}
where $C_e$ is a positive constant independent of $\tau$ and $h$.
\end{lemma}

 \begin{theorem}\label{theorem3}
Let us denote $e^n = u^n-U^n$. Suppose that the solution of problem (\ref{IVP1})-(\ref{IVP2}) is smooth enough
 and {vanishes} as $(x,y)\rightarrow\infty$. Then there exist two small positive constants $\tau_0$ and $h_0$,
 such that, when $\tau<\tau_0$ and $h_x<h_0,\ h_y<h_0$, the numerical solution $U^n$ of the difference
 scheme  \eqref{Boundary}, (\ref{scheme1})--(\ref{scheme2}) and \eqref{scheme3} converges to the exact
 solution $u^n$ in the sense of $L^{\infty}$-norm, with an optimal convergence order $O(\tau^2+h^2_x+h^2_y)$, i.e.,
\begin{align}
\label{estimate1}
\Vert e^n \Vert_{\infty}\leqslant  C_0(\tau^2+h^2_x+h^2_y),  \quad  1\leqslant n \leqslant N,
\end{align}
where $C_0$ is a positive constant independent of $\tau$, $h_x$ and $h_y$.
 \end{theorem}

\begin{proof}
 We use mathematical induction to prove the result. Instead of directly showing (\ref{estimate1}), we will prove a stronger conclusion as follow:
\begin{align}\label{inductionproposition}
\Vert e^n \Vert_2 + \Vert \Lambda^{\alpha} e^n \Vert_2 + \Vert (\mathcal{L}^{\alpha}_x +\mathcal{L}^{\alpha}_y)e^n \Vert_2 \leqslant  C^0(\tau^2 + h_x^2 + h_y^2), \quad  1\leqslant n \leqslant N,
\end{align}
which, once verified, can easily derive the estimate (\ref{estimate1}). {Because it follows directly from Lemma} \ref{Lemma2}, Lemma \ref{Lemma3} and some basic inequalities that
\begin{align}\label{inductionproposition2}
\Vert e^n \Vert_{\infty} \leqslant C_{\alpha}\Vert e^n \Vert_{H^{\alpha}}&\leqslant C_{\alpha}\left(\Vert e^n \Vert_2 + \vert e^n \vert_{H^{\alpha}} + \vert e^n \vert_{H^{\frac{\alpha}{2}}}\right)\\
&\leqslant C_{\alpha}\left( \dfrac{\pi}{2} \right)^{2\alpha}\left(\Vert e^n \Vert_2 + \Vert \Lambda^{\alpha} e^n \Vert_2 + \Vert (\mathcal{L}^{\alpha}_x +\mathcal{L}^{\alpha}_y)e^n \Vert_2\right)\\
&\leqslant C_{\alpha}\left( \dfrac{\pi}{2} \right)^{2\alpha}C^0(\tau^2 + h_x^2 + h_y^2)\\
&:=C_0(\tau^2 + h_x^2 + h_y^2)
\end{align}

Firstly,  it can be proved that (\ref{inductionproposition}) holds for $n=1$. As a matter of fact, from Lemma~\ref{Lemma8}, it holds
\begin{align*}
\|e^{1}\|_2^2=\left(e^1,e^1\right)\leqslant&\; h_xh_y\sum^{M_x-1}_{j=1}\sum^{M_y-1}_{k=1}|e^1_{j,k}|\cdot |e^1_{j,k}|\nonumber\\
\leqslant&\; (b-a)(d-c)\max_{\substack{1\leqslant j\leqslant M_x-1,1\leqslant k\leqslant M_y-1}}|e^1_{j,k}|\cdot |e^1_{j,k}|\\
\leqslant&\; (b-a)(d-c)C_e^2(\tau^2 + h_x^2 + h_y^2)^2,
\end{align*}
\begin{align*}
\|\Lambda^{\alpha} e^{1}\|_2^2=\left((\mathcal{L}^{\alpha}_x+\mathcal{L}^{\alpha}_y)e^1,e^1\right)\leqslant& h_xh_y\sum^{M_x-1}_{j=1}\sum^{M_y-1}_{k=1}|(\mathcal{L}^{\alpha}_x+\mathcal{L}^{\alpha}_y)e^1_{j,k}|\cdot|e^1_{j,k}|\nonumber\\
\leqslant&(b-a)(d-c)\max_{\substack{1\leqslant j\leqslant M_x-1,1\leqslant k\leqslant M_y-1}}|(\mathcal{L}^{\alpha}_x+\mathcal{L}^{\alpha}_y)e^1_{j,k}|\cdot |e^1_{j,k}|\nonumber\\
\leqslant&(b-a)(d-c)C_e^2(\tau^2 + h_x^2 + h_y^2)^2,
\end{align*}
and similarly,
\begin{align*}
\|(\mathcal{L}^{\alpha}_x+\mathcal{L}^{\alpha}_y)e^{1}\|_2^2=&\left((\mathcal{L}^{\alpha}_x+\mathcal{L}^{\alpha}_y)e^1,(\mathcal{L}^{\alpha}_x+\mathcal{L}^{\alpha}_y)e^1\right)\nonumber\\
\leqslant& h_xh_y\sum^{M_x-1}_{j=1}\sum^{M_y-1}_{k=1}|(\mathcal{L}^{\alpha}_x+\mathcal{L}^{\alpha}_y)e^1_{j,k}|^2\nonumber\\
\leqslant&(b-a)(d-c)\max_{1\leqslant j\leqslant M_x-1,1\leqslant k\leqslant M_y-1}|(\mathcal{L}^{\alpha}_x+\mathcal{L}^{\alpha}_y)e^1_{j,k}|^2\nonumber\\
\leqslant&(b-a)(d-c)C_e^2(\tau^2 + h_x^2 + h_y^2)^2.
\end{align*}
Taking $C^0 = 3\sqrt{(b-a)(d-c)}C_e,$  then the {estimate} for $n=1$ {is valid}.

{Now the induction is started}. Assume that  (\ref{inductionproposition}) is true for $m\leqslant n$. We want to show that (\ref{inductionproposition}) is also valid for $n+1$. By the assumption, we have
\begin{align*}
\Vert e^m \Vert_{\infty} \leqslant C_0(\tau^2 + h_x^2 + h_y^2),\quad  1\leqslant m \leqslant n,
\end{align*}
and
  \begin{align}
\|U^{m}\|_{\infty}\leqslant &\; \| u^{m}\|_{\infty}+\| e^{m}\|_{\infty}\nonumber\\
\leqslant &\; C_m + C_0(\tau^2+h^2_x+h^2_y)\nonumber\\
\leqslant &\; C_m + C_0(\tau^2+2h^2_1) \nonumber\\
\leqslant &\;  C_m+1, \quad  1\leqslant m \leqslant n,\label{estimate47}
\end{align}
for $\tau<\tau_1$, $h_x<h_1 $ and $h_y <h_1$, where $\tau_1, h_1$  satisfy that  $\tau_1^2+2h_1^2<\frac{1}{C_0}$. Here,
 $$C_m=\max_{\substack{a\leqslant x \leqslant b,c\leqslant y \leqslant d,0\leqslant t \leqslant T} }|u(x,y,t)|.$$

\noindent \textbf{\textit{Step1.}} Computing the discrete inner product of (\ref{truncation})
with $e^{\bar{n}}$ and taking the real part of the resulting equation, we have
  \begin{align}\label{estimate48}
\frac{\|e^{n+1}\|_2^2-\|e^{n-1}\|_2^2}{4\tau}&=Re\left[(r^n, e^{\bar{n}})\right]-Im\left[(P^n, e^{\bar{n}})\right]\nonumber\\
&\leqslant \|r^n\|_2\|e^{\bar{n}}\|_2+\|P^n\|_2\|e^{\bar{n}}\|_2\nonumber\\
&\leqslant \frac{1}{2}\|r^n\|_2^2+\frac{1}{2}\|P^n\|_2^2+\|e^{\bar{n}}\|_2^2\nonumber\\
&\leqslant \frac{1}{2}\|r^n\|_2^2+\frac{1}{2}\|P^n\|_2^2+\frac{1}{2}\|e^{n+1}\|_2^2+\frac{1}{2}\|e^{n-1}\|_2^2,
\end{align}
where
\begin{align}\label{estimate488}
P^n_{j,k}&=|u^{n}_{j,k}|^2u^{\bar{n}}_{j,k}-|U^{n}_{j,k}|^2U^{\bar{n}}_{j,k}\nonumber\\
     &=(|u^{n}_{j,k}|^2-|U^{n}_{j,k}|^2)u^{\bar{n}}_{j,k}+|U^{n}_{j,k}|^2e^{\bar{n}}_{j,k}\nonumber\\
     &=(|u^{n}_{j,k}|-|U^{n}_{j,k}|)(|u^{n}_{j,k}|+|U^{n}_{j,k}|)u^{\bar{n}}_{j,k}+|U^{n}_{j}|^2e^{\bar{n}}_{j,k}.
\end{align}
{It follows from the assumption} (\ref{estimate47}) that
\begin{align}
\|P^n\|_2\leqslant&\;  (\|u^{n}\|_{\infty}+\|U^{n}\|_{\infty})\|u^{\bar{n}}\|_{\infty}\|e^{n}\|_2+\|U^{n}\|^2_{\infty}\|e^{\bar{n}}\|_2 \nonumber\\
\leqslant&\;  (C_m+C_m+1)C_m\|e^{n}\|_2+(C_m+1)^2\| e^{\bar{n}}\|_2 \nonumber\\
\leqslant&\;  C_{1} (\|e^{n}\|_2+\| e^{\bar{n}}\|_2) \nonumber\\
\leqslant&\;  C_{1} (\|e^{n-1}\|_2+\| e^{n}\|_2+\|e^{n+1}\|_2),\label{estimate49}
\end{align}
for $\tau<\tau_1$ and $h<h_1$, where
\begin{equation*}
  C_{1} = 2(C_m+1)^2.
\end{equation*}
Therefore,
\begin{align}\label{estimate50}
\|P^n\|_2^2
\leqslant& 3C^2_{1} (\|e^{n-1}\|_2^2+\| e^{n}\|_2^2+\|e^{n+1}\|_2^2),
\end{align}
for $\tau<\tau_1$ and $h<h_1$. Following (\ref{estimate48}) and (\ref{estimate50}),
\begin{align}\label{estimate500}
\|e^{n+1}\|_2^2-\|e^{n-1}\|_2^2\leqslant 2\tau\|r^n\|_2^2+(12C^2_{1}+2)\tau(\|e^{n-1}\|_2^2+2\| e^{n}\|_2^2+\|e^{n+1}\|_2^2).
\end{align}

Now {define} $E_{n+1} := \Vert e^{n+1} \Vert_2^2+\Vert e^n \Vert_2^2$, then (\ref{estimate500}) can be simplified into
\begin{align}
E_{n+1}-E_n \leqslant (12C^2_{1}+2)\tau\left( E_{n+1} + E_n \right) + 2\tau\Vert r^n \Vert_2^2.
\end{align}
This is equivalent to
\begin{align}
\left( 1-(12C^2_{1}+2)\tau \right)\left( E_{n+1}-E_n \right)\leqslant (24C^2_{1}+4)\tau E_n +2\tau\Vert r^n \Vert_2^2.
\end{align}
Let $ \tau_1 = \frac{1
}{24C_1^2+4} $, then as $\tau<\tau_1$, we have $1-(12C_1^2+2)\geq \frac{1}{2}$, and
\begin{align}
E_{n+1} - E_n \leqslant (48C_1^2+8)\tau E_n + 4\tau\Vert r^n \Vert_2^2.
\end{align}
Replacing $n$ by $k$ above and summing $k$ over 1 by $n$, then we obtain that
\begin{align}
E_{n+1} \leqslant (48C_1^2+8)\tau \sum_{k=1}^{n}E_k + 4\tau \sum_{k=1}^{n}\Vert r^k \Vert_2^2 + E_1.
\end{align}
Since Lemma~\ref{Lemma8} shows that $ \vert e_{jk}^1 \vert \leqslant C_e \left( \tau^2 + \tau h_x^2 + \tau h_y^2 \right) $ for some positive constant $ C_e $ and $ 1\leqslant j \leqslant M_x-1 $, $ 1\leqslant k \leqslant M_y-1 $, then it apparently holds that
\begin{align}
E_1 = \Vert e^1 \Vert_2^2 = h_x h_y \sum_{j=1}^{M_x-1} \sum_{k=1}^{M_y-1}\vert e_{jk}^1 \vert^2
\leqslant &\; h_x h_y \sum_{j=1}^{M_x-1} \sum_{k=1}^{M_y-1}C_e^2 \left( \tau^2 + h_x^2 + h_y^2 \right)^2 \nonumber\\
\leqslant & \left( b-a \right)\left( d-c \right)C_e^2 \left( \tau^2 + h_x^2 + h_y^2 \right)^2,
\end{align}
where $\tau<1$ is used again. In addition, Lemma~\ref{Lemma7} provides that
\begin{align}
4\tau \sum_{k=1}^{n} \Vert r^k \Vert_2^2 \leqslant 4TC_R^2\left( \tau^2 + h_x^2 + h_y ^2 \right)^2.
\end{align}
This prior estimate can be applied to set a bound:
\begin{align}
E_{n+1} &\leqslant (48C_1^2+8)\tau \sum_{k=1}^n\limits E_k + \left[ 4TC_R^2 + \left( b-a \right)\left( d-c \right)C_e^2 \right]\left( \tau^2 + h_x^2 + h_y^2 \right)^2\nonumber \\
&:= (48C_1^2+8)\tau \sum_{k=1}^n\limits E_k + C_2\left( \tau^2 + h_x^2 + h_y^2 \right)^2,
\end{align}
Now the discrete Gronwall's inequality shows us that
\begin{align}\label{result1}
\Vert e^{n+1} \Vert_2^2 + \Vert e^n \Vert_2^2 = E_{n+1}\leqslant &\; C_2\exp \left\lbrace \sum_{k=1}^{n}(48C_1^2+8)\tau \right\rbrace\left( \tau^2 + h_x^2 + h_y^2 \right)^2\nonumber \\
\leqslant&\; C_2 e^{(48C_1^2+8)T}\left( \tau^2 + h_x^2 + h_y^2 \right)^2\nonumber \\
=&:\;C_3\left( \tau^2 + h_x^2 + h_y^2 \right)^2,
\end{align}

\noindent \textbf{\textit{Step2.}} To carry out further estimations, we compute the discrete inner product of (\ref{truncation}) with $(\mathcal{L}^{\alpha}_x+\mathcal{L}^{\alpha}_y)e^{\bar{n}}$ and consider the real part of the resulting equation. Now we have a new relation
  \begin{align}\label{estimate51}
&\frac{\|\Lambda^{\alpha}e^{n+1}\|_2^2-\|\Lambda^{\alpha}e^{n-1}\|_2^2}{4\tau}\nonumber\\
=&\;Re\left[(r^n,(\mathcal{L}^{\alpha}_x+\mathcal{L}^{\alpha}_y)e^{\bar{n}})\right]-Im\left[(P^n,(\mathcal{L}^{\alpha}_x+\mathcal{L}^{\alpha}_y)e^{\bar{n}})\right]\nonumber\\
\leqslant &\; \|r^n\|_2\|(\mathcal{L}^{\alpha}_x+\mathcal{L}^{\alpha}_y)e^{n+1}\|_2 + \|r^n\|_2\|(\mathcal{L}^{\alpha}_x+\mathcal{L}^{\alpha}_y)e^{n-1}\|_2 + \|P^n\|_2\|(\mathcal{L}^{\alpha}_x+\mathcal{L}^{\alpha}_y)e^{n+1}\|_2\nonumber\\
&\; + \|P^n\|_2\|(\mathcal{L}^{\alpha}_x+\mathcal{L}^{\alpha}_y)e^{n-1}\|_2\nonumber\\
\leqslant&\; \|r^n\|_2^2+\|P^n\|_2^2+\|(\mathcal{L}^{\alpha}_x+\mathcal{L}^{\alpha}_y)e^{n+1}\|_2^2+\|(\mathcal{L}^{\alpha}_x+\mathcal{L}^{\alpha}_y)e^{n-1}\|_2^2.
\end{align}
We should notice what we have concluded at the end of \textbf{\emph{Step1}},
 \begin{align}\label{estimate5151}
 \Vert P^n \Vert^2_2 \leqslant 3C_1^2 \left( \Vert e^{n+1} \Vert^2_2 + 2\Vert e^{n} \Vert^2_2 + \Vert e^{n-1} \Vert^2_2\right) = 3C_1^2\left( E_{n+1} + E_n \right) \leqslant 6C_1^2C_3\left( \tau^2 + h_x^2 + h_y^2 \right)^2.
 \end{align}
%Therefore, $\Vert P^n \Vert$ can be used in the further estimation which is also by means of the Gronwall's inequality. However, we maintain this term momentariy.

Again, we compute the discrete inner product of (\ref{truncation}) with $(\mathcal{L}^{\alpha}_x+\mathcal{L}^{\alpha}_y)\delta_te^{n}$ and analyse the imaginary part of the resulting equation, namely,
\begin{align}
&\;\frac{\|(\mathcal{L}^{\alpha}_{x}+\mathcal{L}^{\alpha}_{y})e^{n+1}\|_2^2-\| (\mathcal{L}^{\alpha}_{x}+\mathcal{L}^{\alpha}_{y})e^{n-1}\|_2^2}{4\tau} \nonumber\\
=&\;-Im\left[(r^n, (\mathcal{L}^{\alpha}_{x}+\mathcal{L}^{\alpha}_{y})\delta_te^{n})\right]-Re\left[(P^n, (\mathcal{L}^{\alpha}_{x}+\mathcal{L}^{\alpha}_{y})\delta_te^{n})\right].\label{estimate55}
\end{align}

Now combining (\ref{estimate51}) and (\ref{estimate55}), we have
\begin{align}\label{estimate 5155}
&\| (\mathcal{L}^{\alpha}_x+\mathcal{L}^{\alpha}_y)e^{n+1}\|_2^2+\|\Lambda^{\alpha} e^{n+1}\|_2^2-\| (\mathcal{L}^{\alpha}_x+\mathcal{L}^{\alpha}_y)e^{n-1}\|_2^2-\|\Lambda^{\alpha} e^{n-1}\|_2^2\nonumber\\
\leqslant&\; 4\tau\left(\|r^n\|_2^2+\|P^n\|_2^2+\|(\mathcal{L}^{\alpha}_x+\mathcal{L}^{\alpha}_y)e^{n+1}\|_2^2+\|(\mathcal{L}^{\alpha}_x+\mathcal{L}^{\alpha}_y)e^{n-1}\|_2^2\right) \nonumber\\
&\;-4\tau Im\left[(r^n, (\mathcal{L}^{\alpha}_{x}+\mathcal{L}^{\alpha}_{y})\delta_te^{n})\right]-4\tau Re\left[(P^n, (\mathcal{L}^{\alpha}_{x}+\mathcal{L}^{\alpha}_{y})\delta_te^{n})\right].
\end{align}
Using the same method as what appears in \textbf{\emph{Step1}}, we  denote that $ F_{n+1} = \Vert(\mathcal{L}^{\alpha}_x+\mathcal{L}^{\alpha}_y)e^{n+1}\Vert_2^2 + \Vert(\mathcal{L}^{\alpha}_x+\mathcal{L}^{\alpha}_y)e^{n}\Vert_2^2 + \Vert \Lambda^{\alpha} e^{n+1} \Vert_2^2 + \Vert \Lambda^{\alpha} e^{n} \Vert_2^2 $, by which the inequality (\ref{estimate 5155}) is reduced to
\begin{align}
F_{n+1}-F_n \leqslant &\; 4\tau\left( F_{n+1} + F_n \right) -4\tau Im\left[(r^n, (\mathcal{L}^{\alpha}_{x}+\mathcal{L}^{\alpha}_{y})\delta_te^{n})\right]\nonumber\\
                      &\;-4\tau Re\left[(P^n, (\mathcal{L}^{\alpha}_{x}+\mathcal{L}^{\alpha}_{y})\delta_te^{n})\right] + 4\tau \left( \Vert r^n \Vert_2^2 + \Vert P^n \Vert_2^2\right).
\end{align}
This can be further estimated as $\tau < \tau_2 =\frac{1}{8}$, that
\begin{align}\label{estimate67}
 F_{n+1}-F_n\leqslant &\; 16\tau F_n  -8\tau Im\left[(r^n, (\mathcal{L}^{\alpha}_{x}+\mathcal{L}^{\alpha}_{y})\delta_te^{n})\right]-8\tau Re\left[(P^n, (\mathcal{L}^{\alpha}_{x}+\mathcal{L}^{\alpha}_{y})\delta_te^{n})\right] \nonumber\\
 &\; + 8\tau \left( \Vert r^n \Vert_2^2 + \Vert P^n \Vert_2^2\right).
\end{align}
We replace the index $n$ by $k$ in (\ref{estimate67}) and sum over $k$ from $1$ to $n$. It yields
\begin{align}\label{estimate68}
F_{n+1}-F_1\leqslant &\; 16\tau \sum_{k=1}^{n}F_k -8\tau \sum_{k=1}^{n} Im\left[(r^k, (\mathcal{L}^{\alpha}_{x}+\mathcal{L}^{\alpha}_{y})\delta_te^{k})\right]\notag\\
                     &\;     -8\tau \sum_{k=1}^{n} Re\left[(P^k, (\mathcal{L}^{\alpha}_{x}+\mathcal{L}^{\alpha}_{y})\delta_te^{k})\right]+ 8\tau \sum_{k=1}^{n} \left( \Vert r^k \Vert_2^2 + \Vert P^k \Vert_2^2\right).
\end{align}

\noindent \textbf{\textit{Step3.}} In order to use the discrete Gronwall's inequality, we need to make some further estimations. Firstly, following from Lemma~\ref{Lemma7} and our previous consequences, we obtain
\begin{equation}\label{estimate69}
8\tau\sum^{n}_{k=1}\left( \Vert r^k \Vert_2^2 + \Vert P^k \Vert_2^2 \right)\leqslant 8T\left(C_R^2 + 6C^2_1C_3\right)(\tau^2+h_x^2+h_y^2)^2 = C_4(\tau^2+h_x^2 +h_y^2)^2.
\end{equation}
Using some basic estimates, we also have that
\begin{align*}
\qquad\qquad&8\tau \sum_{k=1}^{n} Im\left[(r^k,(\mathcal{L}^{\alpha}_{x}+\mathcal{L}^{\alpha}_{y})\delta_te^{k})\right] \notag\\
\qquad\qquad= &16\tau \sum_{k=1}^{n} \dfrac{1}{2\tau}Im \left[ (r^k, (\mathcal{L}^{\alpha}_{x}+\mathcal{L}^{\alpha}_{y})(e^{k+1}-e^{k-1})) \right]\nonumber \\
\qquad\qquad= &8\tau \sum_{k=1}^{n} \dfrac{1}{2\tau}Im \left[ (r^k, (\mathcal{L}^{\alpha}_{x}+\mathcal{L}^{\alpha}_{y})e^{k+1}) \right]-8\tau \sum_{k=1}^{n} \dfrac{1}{2\tau}Im \left[ (r^k, (\mathcal{L}^{\alpha}_{x}+\mathcal{L}^{\alpha}_{y})e^{k-1}) \right]\nonumber\\
\qquad\qquad= &8\tau \sum_{k=2}^{n+1} \dfrac{1}{2\tau}Im \left[( r^{k-1}, (\mathcal{L}^{\alpha}_{x}+\mathcal{L}^{\alpha}_{y})e^{k}) \right]-8\tau \sum_{k=0}^{n-1} \dfrac{1}{2\tau}Im \left[ (r^{k+1}, (\mathcal{L}^{\alpha}_{x}+\mathcal{L}^{\alpha}_{y})e^{k}) \right]\nonumber\\
\qquad\qquad= &4Im \left[ (r^n,(\mathcal{L}^{\alpha}_{x}+\mathcal{L}^{\alpha}_{y})e^{n+1})+(r^{n-1},(\mathcal{L}^{\alpha}_{x}+\mathcal{L}^{\alpha}_{y})e^n)
-(r^1,(\mathcal{L}^{\alpha}_{x}+\mathcal{L}^{\alpha}_{y})e^0)\right.\nonumber\\
&\left.-(r^2,(\mathcal{L}^{\alpha}_{x}+\mathcal{L}^{\alpha}_{y})e^1)\right]-8\tau\sum_{k=2}^{n-1} Im\left[ (\delta_t r^k,(\mathcal{L}^{\alpha}_{x}+\mathcal{L}^{\alpha}_{y})e^k) \right].
\end{align*}
It follows from Lemma~\ref{Lemma7} and Lemma~\ref{Lemma8}, after some simple calculations,
\begin{align}
\qquad\qquad&8\tau \sum_{k=1}^{n} Im\left[(r^k,(\mathcal{L}^{\alpha}_{x}+\mathcal{L}^{\alpha}_{y})\delta_te^{k})\right] \nonumber \\
\qquad\qquad\leqslant & 4\left[  \|r^n\|_2 \cdot \|(\mathcal{L}^{\alpha}_{x}+\mathcal{L}^{\alpha}_{y})e^{n+1}\|_2 + \|r^{n-1} \|_2
\cdot \| (\mathcal{L}^{\alpha}_{x}+\mathcal{L}^{\alpha}_{y})e^{n} \|_2 \right.\nonumber\\
\qquad\qquad&\left.+\| r^2\|_2\cdot\|(\mathcal{L}^{\alpha}_{x}+\mathcal{L}^{\alpha}_{y})e^1 \|_2  \right] +8\tau \sum_{k=2}^{n-1}\| \delta_t r^k \|_2 \cdot \| (\mathcal{L}^{\alpha}_{x}+\mathcal{L}^{\alpha}_{y})e^k \|_2 \nonumber\\
\qquad\qquad\leqslant &4\left[ 4\| r^n \|_2^2 + \dfrac{1}{16}\| (\mathcal{L}^{\alpha}_{x}+\mathcal{L}^{\alpha}_{y})e^{n+1} \|_2^2 + 4\| r^{n-1} \|_2^2 + \dfrac{1}{16}\| (\mathcal{L}^{\alpha}_{x}+\mathcal{L}^{\alpha}_{y})e^{n} \|_2^2 + \dfrac{1}{2}\| r^2 \|_2^2 \right.\nonumber\\
&\left. + \dfrac{1}{2}\| (\mathcal{L}^{\alpha}_{x}+\mathcal{L}^{\alpha}_{y})e^1 \|_2^2\right]+ 4\tau \sum_{k=2}^{n-1}\| \delta_t r^k \|_2^2 + 4\tau \sum_{k=2}^{n-1} \| (\mathcal{L}^{\alpha}_{x}+\mathcal{L}^{\alpha}_{y})e^k \|_2^2\nonumber \\
\qquad\qquad= &16\| r^n\|_2^2 + 16 \| r^{n-1} \|_2^2 + 2\| r^2 \|_2^2 + 2\| (\mathcal{L}^{\alpha}_{x}+\mathcal{L}^{\alpha}_{y})e^1 \|_2 + \dfrac{1}{4}\| (\mathcal{L}^{\alpha}_{x}+\mathcal{L}^{\alpha}_{y})e^{n+1} \|_2^2\nonumber \\
\qquad\qquad& + \dfrac{1}{4}\| (\mathcal{L}^{\alpha}_{x}+\mathcal{L}^{\alpha}_{y})e^{n} \|_2^2+ 4\tau \sum_{k=2}^{n-1}\| \delta_t r^k \|_2^2 + 4\tau \sum_{k=2}^{n-1} \| (\mathcal{L}^{\alpha}_{x}+\mathcal{L}^{\alpha}_{y})e^k \|_2^2\nonumber \\
\qquad\qquad\leqslant & 4\tau \sum_{k=2}^{n-1} \| (\mathcal{L}^{\alpha}_{x}+\mathcal{L}^{\alpha}_{y})e^k \|_2^2 + \dfrac{1}{4}\| (\mathcal{L}^{\alpha}_{x}+\mathcal{L}^{\alpha}_{y})e^{n+1} \|_2^2 + \dfrac{1}{4}\| (\mathcal{L}^{\alpha}_{x}+\mathcal{L}^{\alpha}_{y})e^{n} \|_2^2 \nonumber\\
&+ C_5\left(\tau^2 + h_x^2 + h_y^2 \right)^2,\label{estimation450}
\end{align}
Since $ \Vert r^n \Vert_2^2 $ and $\Vert \delta_t r^n \Vert_2^2$ are all bounded by a constant-multiple of $\left(\tau^2 + h_x^2 + h_y^2 \right)^2$,  we can use $C_5$ as a general positive constant to represent the sum of all coefficients of $\left(\tau^2 + h_x^2 + h_y^2 \right)^2$ in the above estimation.

Next we need to estimate $8\tau\sum_{k=1}^{n}\limits Re\left[(P^k, (\mathcal{L}^{\alpha}_{x}+\mathcal{L}^{\alpha}_{y})\delta_te^{k})\right]$. We notice that
\begin{align}
&8\tau \sum_{k=1}^{n} Re\left[(P^k, (\mathcal{L}^{\alpha}_{x}+\mathcal{L}^{\alpha}_{y})\delta_te^{k})\right]\nonumber \\
= &4 Re\left( P^n, \left( \mathcal{L}_x^{\alpha} + \mathcal{L}_y^{\alpha} \right)e^{n+1} \right) + Re \left( P^{n-1}, \left( \mathcal{L}_x^{\alpha} + \mathcal{L}_y^{\alpha} \right)e^n \right) - Re\left( P^1, \left( \mathcal{L}_x^{\alpha} + \mathcal{L}_y^{\alpha} \right)e^0 \right) \nonumber\\
& -Re\left( P^2, \left( \mathcal{L}_x^{\alpha} + \mathcal{L}_y^{\alpha} \right)e^1 \right) -8\tau \sum_{k=2}^{n-1}\limits Re (\delta_t P^k,e^k),
\end{align}
where
\begin{align}
P^k=&|u^{k}|^2u^{\bar{k}}-|U^{k}|^2U^{\bar{k}}\nonumber \\
=& u^k (u^k)^*u^{\bar{k}} - \left( u^k-e^k \right)\left( (u^k)^*-(e^k)^* \right) \left( u^{\bar{k}}-e^{\bar{k}} \right)\nonumber \\
=&(u^k)^{*}u^{\bar{k}}e^k+u^ku^{\bar{k}}(e^k)^*+u^k(u^k)^{*}e^{\bar{k}}-(e^k)^*e^{\bar{k}}u^k-e^ke^{\bar{k}}(u^k)^{*}\nonumber\\
&-e^k(e^k)^*u^{\bar{k}}+e^k(e^k)^*e^{\bar{k}}
\end{align}
and
\begin{align}
\delta_t P^k=&\delta_t((u^k)^*u^{\bar{k}}e^k)+\delta_t (u^ku^{\bar{k}}(e^k)^*)+\delta_t (u^k(u^k)^*e^{\bar{k}})-\delta_t ((e^k)^*e^{\bar{k}}u^k)\nonumber\\
&-\delta_t (e^ke^{\bar{k}}(u^k)^*)-\delta_t(e^k(e^k)^*u^{\bar{k}})+\delta_t(e^k(e^k)^*e^{\bar{k}})
\triangleq \sum^{7}_{j=1}I_j.
\end{align}

In order to estimate $\delta_t e^{k}$ and similar terms which shows in the expansion of $\delta_t P^k$, we first use the truncation error (\ref{truncation}) to find
\begin{align}\label{delta en}
\Vert \delta_t e^k \Vert_2 \leqslant &\; \Vert r^k \Vert_2 + \Vert \mathbf{i} P^k \Vert_2 + \Vert \mathbf{i}(\mathcal{L}^{\alpha}_x+\mathcal{L}^{\alpha}_y)e^{\bar{k}} \Vert_2 \nonumber\\
\leqslant &\;\Vert r^k \Vert_2 + \Vert P^k \Vert_2 + \Vert (\mathcal{L}^{\alpha}_x+\mathcal{L}^{\alpha}_y)e^{k+1}\Vert_2 + \Vert (\mathcal{L}^{\alpha}_x+\mathcal{L}^{\alpha}_y)e^{k-1}\Vert_2.
\end{align}

 Next, we observe that the exact solution $u$ is sufficiently smooth, so that $\Vert u \Vert_{\infty}$ and $\Vert u_t \Vert_{\infty}$ must be bounded in the domain $[0,T]\times[a,b] \times [c,d]$. In addition, we may as well take use of the result in \textbf{\emph{Step1}}, i.e., $ \Vert e^{n+1} \Vert_2 + \Vert e^{n} \Vert_2  \leqslant \sqrt{3C_3}(\tau^2 + h_x^2 + h_y^2) < C $ as $\tau,h_x,h_y <1$. Therefore, combining with Lemma~\ref{Lemmadelta}, the seven items can be estimated as follows:

\begin{flalign}
\begin{split}
\hspace{4mm}
\Vert I_1 \Vert_2 &= \Vert \delta_t (u^k)^* e^{\bar{k}} u^{\bar{\bar{k}}} + \delta_t e^k ( u^{\bar{k}} )^*u^{\bar{\bar{k}}} + \delta_t u^{\bar{k}}(u^* e)^{\bar{k}}\Vert_2\nonumber \\
&\leqslant \Vert u_t \Vert_{\infty}\cdot\Vert u \Vert_{\infty}\cdot\Vert e^{\bar{k}} \Vert_2 + \Vert u \Vert^2_{\infty}\cdot\Vert \delta_t e^k \Vert_2 + \Vert u_t \Vert_{\infty} \Vert u \Vert_{\infty}\left( \Vert e^{k+1} \Vert_2 + \Vert e^{k-1} \Vert_2 \right)\nonumber \\
&\leqslant C\left( \Vert e^{k+1} \Vert_2 + \Vert e^{k-1} \Vert_2 + \Vert \delta_t e^k \Vert_2 \right)\nonumber \\
&\leqslant C\left( \Vert e^{k+1} \Vert_2 + \Vert e^{k-1} \Vert_2 + \Vert r^{k} \Vert_2 + \Vert P^{k} \Vert_2 + \Vert (\mathcal{L}^{\alpha}_x+\mathcal{L}^{\alpha}_y)e^{k+1}\Vert_2 + \Vert (\mathcal{L}^{\alpha}_x+\mathcal{L}^{\alpha}_y)e^{k-1}\Vert_2 \right),
\end{split}&
\end{flalign}
\begin{flalign}
\begin{split}
\hspace{4mm}
\Vert I_2 \Vert_2 &= \Vert \delta_t u^k ( e^{\bar{k}} )^* u^{\bar{\bar{k}}} + \delta_t (e^k)^*u^{\bar{k}}u^{\bar{\bar{k}}} + \delta_t u^{\bar{k}} (ue^*)^{\bar{k}} \Vert_2\nonumber \\
&\leqslant \Vert u_t \Vert_{\infty}\cdot\Vert u \Vert_{\infty}\cdot\Vert e^{\bar{k}} \Vert_2 + \Vert u \Vert^2_{\infty}\cdot\Vert \delta_t e^k \Vert_2 + \Vert u_t \Vert_{\infty} \Vert u \Vert_{\infty}\left( \Vert e^{k+1} \Vert_2 + \Vert e^{k-1} \Vert_2 \right)\nonumber \\
&\leqslant C\left( \Vert e^{k+1} \Vert_2 + \Vert e^{k-1} \Vert_2 + \Vert \delta_t e^k \Vert_2 \right)\nonumber \\
&\leqslant C\left( \Vert e^{k+1} \Vert_2 + \Vert e^{k-1} \Vert_2 + \Vert r^{k} \Vert_2 + \Vert P^{k} \Vert_2 + \Vert (\mathcal{L}^{\alpha}_x+\mathcal{L}^{\alpha}_y)e^{k+1}\Vert_2 + \Vert (\mathcal{L}^{\alpha}_x+\mathcal{L}^{\alpha}_y)e^{k-1}\Vert_2 \right),
\end{split}&
\end{flalign}
\begin{flalign}
\begin{split}
\hspace{4mm}
\Vert I_3 \Vert_2  &= \Vert \delta_t u^k(u^k)^* e^{\bar{\bar{k}}} + \delta_t (u^k)^*u^{\bar{k}}e^{\bar{\bar{k}}} + (uu^*)^{\bar{k}}\delta_t e^{\bar{k}} \Vert_2 \nonumber \\
&\leqslant 2 \Vert u_t \Vert_{\infty}\cdot \Vert  u \Vert_{\infty}\cdot \Vert e^{\bar{\bar{k}}} \Vert_2  + \Vert   uu^* \Vert_\infty\cdot \Vert \delta_t e^{\bar{k}} \Vert_2 \nonumber \\
&\leqslant C\left(\Vert e^{k+2} \Vert_2 + \|e^k\|_2+ \Vert e^{k-2} \Vert_2 + \Vert \delta_t e^{k+1} \Vert_2 + \Vert \delta_t e^{k-1} \Vert_2 \right)\nonumber \\
&\leqslant C\left( \Vert e^{k+2} \Vert_2 +\|e^k\|_2+ \Vert e^{k-2} \Vert_2 + \Vert r^{k+1} \Vert_2 + \Vert r^{k-1} \Vert_2 + \Vert P^{k+1} \Vert_2 + \Vert P^{k-1} \Vert_2\right.\nonumber\\
 &\quad \left.+ \Vert (\mathcal{L}^{\alpha}_x+\mathcal{L}^{\alpha}_y)e^{k+2}\Vert_2 + \Vert (\mathcal{L}^{\alpha}_x+\mathcal{L}^{\alpha}_y)e^{k}\Vert_2 + \Vert (\mathcal{L}^{\alpha}_x+\mathcal{L}^{\alpha}_y)e^{k-2}\Vert_2 \right),
\end{split}&
\end{flalign}

Especially for $2\leqslant k \leqslant n-1$, the estimation (\ref{delta en}) only involves values from the first $n$-th time layers, thus by the induction assumption, we have $\Vert \delta_t e^k \Vert_2 \leqslant C$ when $\tau, h_x, h_y<1$. By using the estimate (\ref{inductionproposition2}) under the induction assumption  and Lemma~\ref{Lemma2}, we have
\begin{flalign}
\begin{split}
\hspace{4mm}
\Vert I_4 \Vert_2 &= \Vert\delta_t u^k (e^{\bar{k}})^*e^{\bar{\bar{k}}} + \delta_t (e^k)^*u^{\bar{k}}e^{\bar{\bar{k}}} + \delta_t e^{\bar{k}}(ue^*)^{\bar{k}}\Vert_2\nonumber \\
&\leqslant \Vert u_t \Vert_{\infty}\Vert e^{\bar{k}} \Vert_{\infty}\Vert e^{\bar{\bar{k}}} \Vert_2 + \Vert u \Vert_{\infty}\cdot \Vert \delta_t e^k \Vert_{2}\cdot\Vert e^{\bar{\bar{k}}}\Vert_{\infty} + \Vert u \Vert_{\infty}\cdot\left(\Vert e^{k+1} \Vert_{\infty} +\Vert e^{k-1} \Vert_{\infty} \right)\cdot\Vert \delta_t e^{\bar{k}} \Vert_2 \nonumber\\
&\leqslant C\left(\Vert e^{\bar{\bar{k}}} \Vert_2+\Vert e^{\bar{\bar{k}}}\Vert_{H^{\alpha}}+\Vert \delta_t e^{\bar{k}} \Vert_2\right)\\
&\leqslant C\left( \Vert e^{k+2} \Vert_2 + \Vert e^{k} \Vert_2 + \Vert e^{k-2} \Vert_2 + \Vert \Lambda^{\alpha}e^{k+2} \Vert_2 + \Vert \Lambda^{\alpha}e^{k} \Vert_2 + \Vert \Lambda^{\alpha}e^{k-2} \Vert_2 \right)\nonumber \\
&\quad+ C\left(\Vert (\mathcal{L}^{\alpha}_x+\mathcal{L}^{\alpha}_y) e^{k+2} \Vert_2 + \Vert (\mathcal{L}^{\alpha}_x+\mathcal{L}^{\alpha}_y) e^{k} \Vert_2 + \Vert (\mathcal{L}^{\alpha}_x+\mathcal{L}^{\alpha}_y) e^{k-2} \Vert_2 + \Vert \delta_t e^{k+1} \Vert_2 + \Vert \delta_t e^{k-1} \Vert_2 \right)\nonumber \\
&\leqslant C\left( \Vert e^{k+2} \Vert_2 + \Vert e^{k} \Vert_2 + \Vert e^{k-2} \Vert_2 + \Vert r^{k+1} \Vert_2 + \Vert r^{k-1} \Vert_2 + \Vert P^{k+1} \Vert_2 + \Vert P^{k-1} \Vert_2\right)\nonumber \\
&\quad + C\left(\Vert (\mathcal{L}^{\alpha}_x+\mathcal{L}^{\alpha}_y)e^{k+2}\Vert_2 + \Vert (\mathcal{L}^{\alpha}_x+\mathcal{L}^{\alpha}_y)e^{k}\Vert_2 + \Vert (\mathcal{L}^{\alpha}_x+\mathcal{L}^{\alpha}_y)e^{k-2}\Vert_2 \right),
\end{split}&
\end{flalign}
\begin{flalign}
\begin{split}
\hspace{4mm}
\Vert I_5 \Vert_2 &= \Vert \delta_t (u^k)^*e^{\bar{k}}e^{\bar{\bar{k}}} + \delta_t e^k (u^{\bar{k}})^*e^{\bar{\bar{k}}} + \delta_t e^{\bar{k}}(u^*e)^{\bar{k}} \Vert_2\nonumber \\
&\leqslant \Vert u_t \Vert_{\infty} \Vert e^{\bar{k}} \Vert_{\infty}\Vert e^{\bar{\bar{k}}} \Vert_{2} + \Vert \delta_t e^k \Vert_2 \Vert u \Vert_{\infty} \Vert e^{\bar{\bar{k}}} \Vert_2 + \Vert u \Vert_{\infty}\left( \Vert e^{k+1} \Vert_{\infty} + \Vert e^{k-1} \Vert_{\infty}\right)\Vert \delta_t e^{\bar{k}} \Vert_2\nonumber \\
&\leqslant C\left(\Vert e^{\bar{\bar{k}}} \Vert_2+ \Vert \delta_t e^{\bar{k}} \Vert_2 \right)\nonumber \\
&\leqslant C\left( \Vert e^{k+2} \Vert_2 + \Vert e^{k} \Vert_2 + \Vert e^{k-2} \Vert_2 + \Vert \delta_t e^{k+1} \Vert_2 + \Vert \delta_t e^{k-1} \Vert_2 \right)\nonumber \\
&\leqslant C\left( \Vert e^{k+2} \Vert_2 + \Vert e^{k} \Vert_2 + \Vert e^{k-2} \Vert_2 + \Vert r^{k+1} \Vert_2 + \Vert r^{k-1} \Vert_2 + \Vert P^{k+1} \Vert_2 + \Vert P^{k-1} \Vert_2\right)\nonumber \\
&\quad + C\left(\Vert (\mathcal{L}^{\alpha}_x+\mathcal{L}^{\alpha}_y)e^{k+2}\Vert_2 + \Vert (\mathcal{L}^{\alpha}_x+\mathcal{L}^{\alpha}_y)e^{k}\Vert_2 + \Vert (\mathcal{L}^{\alpha}_x+\mathcal{L}^{\alpha}_y)e^{k-2}\Vert_2 \right),
\end{split}&
\end{flalign}
\begin{flalign}
\begin{split}
\hspace{4mm}
\Vert I_6 \Vert_2 &= \Vert \delta_t e^k (e^{\bar{k}})^*u^{\bar{\bar{k}}} + \delta_t (e^k)^*e^{\bar{k}}u^{\bar{\bar{k}}} + \delta_t u^{\bar{k}}(ee^*)^{\bar{k}} \Vert_2\nonumber \\
&\leqslant 2\Vert u \Vert_{\infty} \Vert e^{\bar{k}} \Vert_{\infty} \Vert \delta_t e^k \Vert_2 + \Vert u_t \Vert_{\infty}\left( \Vert e^{k+1} \Vert_{\infty}\Vert e^{k+1} \Vert_{2} + \Vert e^{k-1} \Vert_{\infty}\Vert e^{k-1} \Vert_{2} \right)\nonumber \\
&\leqslant C\left( \Vert e^{k+1} \Vert_2 + \Vert e^{k-1} \Vert_2 + \Vert \delta_t e^k \Vert_2 \right)\nonumber \\
&\leqslant C\left( \Vert e^{k+1} \Vert_2 + \Vert e^{k-1} \Vert_2 + \Vert r^{k} \Vert_2 + \Vert P^{k} \Vert_2 + \Vert (\mathcal{L}^{\alpha}_x+\mathcal{L}^{\alpha}_y)e^{k+1}\Vert_2 + \Vert (\mathcal{L}^{\alpha}_x+\mathcal{L}^{\alpha}_y)e^{k-1}\Vert_2 \right),
\end{split}&
\end{flalign}
\begin{align}
\begin{split}
\hspace{4mm}
\Vert I_7 \Vert_2 &= \Vert \delta_t e^k(e^{\bar{k}})^*e^{\bar{\bar{k}}} + \delta_t (e^k)^* e^{\bar{k}}e^{\bar{\bar{k}}} + \delta_te^{\bar{k}}(ee^*)^{\bar{k}}\Vert_2\nonumber \\
&\leqslant 2\Vert \delta_t e^k \Vert_2 \Vert e^{\bar{k}} \Vert_{\infty}\Vert e^{\bar{\bar{k}}} \Vert_{\infty} + \Vert \delta_t e^{\bar{k}} \Vert_2\left( \Vert e^{k+1} \Vert^2_{\infty} + \Vert e^{k-1} \Vert^2_{\infty} \right) \nonumber \\
&\leqslant C\left(\Vert e^{\bar{\bar{k}}} \Vert_{\infty}+\Vert \delta_t e^{\bar{k}} \Vert_2\right)\nonumber \\
&\leqslant C\left(\Vert e^{\bar{\bar{k}}} \Vert_{H^{\alpha}}+\Vert \delta_t e^{\bar{k}} \Vert_2\right)\nonumber \\
&\leqslant C\left(\Vert e^{k+2} \Vert_{H^{\alpha}} + \Vert e^k \Vert_{H^{\alpha}} + \Vert e^{k-2} \Vert_{H^{\alpha}} + \Vert \delta_t e^{k+1} \Vert_2 + \Vert \delta_t e^{k-1} \Vert_2\right)\nonumber \\
&\leqslant C\left(\Vert e^{k+2} \Vert_2 + \Vert e^k \Vert_2 + \Vert e^{k-2}\Vert_2 + \Vert \Lambda^{\alpha} e^{k+2} \Vert_2 + \Vert \Lambda^{\alpha} e^k \Vert_2 + \Vert \Lambda^{\alpha} e^{k-2} \Vert_2 \right)\nonumber \\
&\quad +C\left( \Vert (\mathcal{L}^{\alpha}_x+\mathcal{L}^{\alpha}_y)e^{k+2}\Vert_2 + \Vert (\mathcal{L}^{\alpha}_x+\mathcal{L}^{\alpha}_y)e^{k}\Vert_2 + \Vert (\mathcal{L}^{\alpha}_x+\mathcal{L}^{\alpha}_y)e^{k-2}\Vert_2\right.\nonumber\\
&\quad\left. + \Vert r^{k+1} \Vert_2 + \Vert r^{k-1} \Vert_2+ \Vert P^{k+1} \Vert_2 + \Vert P^{k-1} \Vert_2\right),
\end{split}&
\end{align}
 when $\tau, h_x, h_y<1$.

Now we have for $2\leqslant k \leqslant n-1$,
\begin{align}
\qquad\qquad&\left|\left( \delta_t P^k, (\mathcal{L}^{\alpha}_{x}+\mathcal{L}^{\alpha}_{y})e^{k}\right)\right| \nonumber\\
\qquad\qquad\leqslant & \sum_{i=1}^{7}\Vert I_i \Vert_2\Vert (\mathcal{L}^{\alpha}_{x}+\mathcal{L}^{\alpha}_{y})e^{k} \Vert_2 \nonumber \\
\qquad\qquad\leqslant & C\left( \Vert e^{k+2} \Vert_2^2 + \Vert e^{k+1} \Vert_2^2 + \Vert e^k \Vert_2^2 + \Vert e^{k-1} \Vert_2^2 + \Vert e^{k-2} \Vert_2^2\right)\nonumber \\
\qquad\qquad&+ C\left( \Vert \Lambda^{\alpha}e^{k+2} \Vert_2^2 + \Vert \Lambda^{\alpha}e^{k+1} \Vert_2^2 + \Vert \Lambda^{\alpha}e^{k} \Vert_2^2 + \Vert \Lambda^{\alpha}e^{k-1} \Vert_2^2 + \Vert \Lambda^{\alpha}e^{k-2} \Vert_2^2\right)\nonumber \\
\qquad\qquad&+ C\left( \Vert (\mathcal{L}^{\alpha}_{x}+\mathcal{L}^{\alpha}_{y})e^{k+2} \Vert_2^2 \right.\nonumber\\
\qquad\qquad &\left.+ \Vert (\mathcal{L}^{\alpha}_{x}+\mathcal{L}^{\alpha}_{y})e^{k+1} \Vert_2^2 + \Vert (\mathcal{L}^{\alpha}_{x}+\mathcal{L}^{\alpha}_{y})e^{k} \Vert_2^2 + \Vert (\mathcal{L}^{\alpha}_{x}+\mathcal{L}^{\alpha}_{y})e^{k-1} \Vert_2^2 + \Vert (\mathcal{L}^{\alpha}_{x}+\mathcal{L}^{\alpha}_{y})e^{k-2} \Vert_2^2\right)\nonumber \\
\qquad\qquad& + C\left(\Vert r^{k+1}\Vert_2^2 + \Vert r^k \Vert_2^2 + \Vert r^{k-1} \Vert_2^2 + \Vert P^{k+1} \Vert_2^2  + \Vert P^k \Vert_2^2 + \Vert P^{k-1} \Vert_2^2\right).
\end{align}
Therefore, it follows from Lemma~\ref{Lemma8} and the estimation (\ref{estimate5151}), that
\begin{align}\label{estimation455}
&\quad 8\tau \sum_{k=1}^{n} Re\left[(P^k, (\mathcal{L}^{\alpha}_{x}+\mathcal{L}^{\alpha}_{y})\delta_te^{k})\right]\nonumber \\
&= 4[ Re\left( P^n, \left( \mathcal{L}_x^{\alpha} + \mathcal{L}_y^{\alpha} \right)e^{n+1} \right) + Re \left( P^{n-1}, \left( \mathcal{L}_x^{\alpha} + \mathcal{L}_y^{\alpha} \right)e^n \right) - Re\left( P^1, \left( \mathcal{L}_x^{\alpha} + \mathcal{L}_y^{\alpha} \right)e^0 \right) \nonumber \\
&\quad -Re\left( P^2, \left( \mathcal{L}_x^{\alpha} + \mathcal{L}_y^{\alpha} \right)e^1 \right) ]- 8\tau Re \sum_{k=2}^{n-1}\left( \delta_t P^k, \left( \mathcal{L}_x^{\alpha} + \mathcal{L}_y^{\alpha} \right)e^k \right)\nonumber \\
&\leqslant 8\tau\sum_{k=2}^{n-1} \vert\left( \delta_t P^k, (\mathcal{L}^{\alpha}_{x}+\mathcal{L}^{\alpha}_{y})e^{k}\right)\vert + 4\left( 4\Vert P^n \Vert_2^2 + 4\Vert P^{n-1} \Vert_2^2 + \dfrac{1}{16}\Vert \left( \mathcal{L}_x^{\alpha} + \mathcal{L}_y^{\alpha} \right)e^{n+1} \Vert_2^2\right)\nonumber \\
&\quad+ 4\left(\dfrac{1}{16}\Vert \left( \mathcal{L}_x^{\alpha} + \mathcal{L}_y^{\alpha} \right)e^{n} \Vert_2^2 + \left( 6C_1^2C_3 + C_e^2 \right)\left( \tau^2 + h_x^2 + h_y^2 \right)^2 \right)\nonumber \\
&\leqslant 8C\tau \sum_{k=2}^{n-1}(\Vert (\mathcal{L}^{\alpha}_{x}+\mathcal{L}^{\alpha}_{y})e^{k+2} \Vert_2^2 + \Vert (\mathcal{L}^{\alpha}_{x}+\mathcal{L}^{\alpha}_{y})e^{k+1} \Vert_2^2 + \Vert (\mathcal{L}^{\alpha}_{x}+\mathcal{L}^{\alpha}_{y})e^{k} \Vert_2^2 \nonumber \\
&\quad + \Vert (\mathcal{L}^{\alpha}_{x}+\mathcal{L}^{\alpha}_{y})e^{k-1} \Vert_2^2 + \Vert (\mathcal{L}^{\alpha}_{x}+\mathcal{L}^{\alpha}_{y})e^{k-2} \Vert_2^2)\nonumber\\
&\quad + 8C\tau \sum_{k=2}^{n-1}\left(\Vert \Lambda^{\alpha} e^{k+2} \Vert_2^2 + \Vert \Lambda^{\alpha}e^{k+1} \Vert_2^2 + \Vert \Lambda^{\alpha}e^{k} \Vert_2^2 + \Vert\Lambda^{\alpha}e^{k-1} \Vert_2^2 + \Vert \Lambda^{\alpha}e^{k-2} \Vert_2^2\right)\nonumber\\
&\quad + 8C\tau \sum_{k=2}^{n-1}\left( \Vert e^{k+2} \Vert_2^2 + \Vert e^{k+1} \Vert_2^2 + \Vert e^k \Vert_2^2 + \Vert e^{k-1} \Vert_2^2 + \Vert e^{k-2} \Vert_2^2\right)\nonumber \\
&\quad + 8C\tau\sum_{k=2}^{n-1}\left( \Vert r^{k+1} \Vert_2^2 + \Vert r^{k} \Vert_2^2 + \Vert r^{k-1} \Vert_2^2 + \Vert P^{k+1} \Vert_2^2 + \Vert P^{k} \Vert_2^2 + \Vert P^{k-1} \Vert_2^2\right)\nonumber\\
&\quad+8\tau C\left( \Vert e^{k+2} \Vert_2^2 + \Vert e^{k+1} \Vert_2^2 + \Vert e^k \Vert_2^2 + \Vert e^{k-1} \Vert_2^2 + \Vert e^{k-2} \Vert_2^2\right)\nonumber\\
&\quad + 16\Vert P^n \Vert_2^2 + 16\Vert P^{n-1} \Vert_2^2 + \dfrac{1}{4}\Vert (\mathcal{L}^{\alpha}_{x}+\mathcal{L}^{\alpha}_{y})e^{n+1} \Vert_2^2 + \dfrac{1}{4}\Vert (\mathcal{L}^{\alpha}_{x}+\mathcal{L}^{\alpha}_{y})e^{n} \Vert_2^2\nonumber \\
&\quad + 4\left(6C_1^2C_3 + C_e^2 \right)\left(\tau^2 + h_x^2 + h_y^2\right)^2\nonumber \\
&\leqslant 40C\tau\sum_{k=1}^{n}\left(\Vert (\mathcal{L}^{\alpha}_{x}+\mathcal{L}^{\alpha}_{y})e^{k} \Vert_2^2 + \Vert \Lambda^{\alpha} e^k \Vert_2^2 \right) + \left(8C\tau+\dfrac{1}{4}\right)\Vert(\mathcal{L}^{\alpha}_{x}+\mathcal{L}^{\alpha}_{y})e^{n+1}\Vert_2^2 \nonumber \\
&\quad + \left(8C\tau+\dfrac{1}{4}\right)\Vert(\mathcal{L}^{\alpha}_{x}+\mathcal{L}^{\alpha}_{y})e^{n}\Vert_2^2+ 8C\tau \left( \Vert \Lambda^{\alpha}e^{n+1} \Vert_2^2 + \Vert \Lambda^{\alpha}e^n \Vert_2^2 \right)  \nonumber\\
&\quad   + C_6\left( \tau^2 + h_x^2 + h_y^2 \right)^2,
\end{align}
Since $ \Vert e^{n+1} \Vert_2^2 + \Vert e^n \Vert_2^2 $, $\Vert P^n \Vert_2^2$ and $\Vert r^n \Vert_2^2$ are all bounded by constant multiples of $\left(\tau^2 + h_x^2 + h_y^2 \right)^2$, we can use $C_6$ as a general positive constant to represent the sum of all coefficients of $\left(\tau^2 + h_x^2 + h_y^2 \right)^2$ in the above estimation.

Now we put these estimations (\ref{estimation450})-(\ref{estimation455}) backward into the inequality (\ref{estimate68}), to obtain
\begin{align}
\qquad\qquad&F_{n+1}-F_1\nonumber\\
\qquad\qquad\leqslant &16\tau\sum_{k=1}^{n}F_k + \left( 40C\tau + 4\tau \right)\sum_{k=1}^{n}\left(\Vert (\mathcal{L}^{\alpha}_{x}+\mathcal{L}^{\alpha}_{y})e^{k} \Vert_2^2 + \Vert \Lambda^{\alpha} e^k \Vert_2^2 \right) \nonumber\\
\qquad\qquad&+ \left(8C\tau + \dfrac{1}{2}\right)\left(\Vert(\mathcal{L}^{\alpha}_{x}+\mathcal{L}^{\alpha}_{y})e^{n+1}\Vert_2^2 + \Vert(\mathcal{L}^{\alpha}_{x}+\mathcal{L}^{\alpha}_{y})e^{n}\Vert_2^2\right)\nonumber \\
\qquad\qquad& + 8C\tau \left( \Vert \Lambda^{\alpha}e^{n+1} \Vert_2^2 + \Vert \Lambda^{\alpha}e^n \Vert_2^2 \right) + \left(C_4+C_5\right)\left( \tau^2 + h_x^2 + h_y^2 \right)^2\nonumber \\
\leqslant& \left( 40C\tau + 20\tau \right)\sum_{k=1}^{n}F_k + \left(8C\tau + \dfrac{1}{2}\right)\left(\Vert(\mathcal{L}^{\alpha}_{x}+\mathcal{L}^{\alpha}_{y})e^{n+1}\Vert_2^2 + \Vert(\mathcal{L}^{\alpha}_{x}+\mathcal{L}^{\alpha}_{y})e^{n}\Vert_2^2\right)\nonumber\\
\qquad\qquad&+ 8C\tau \left( \Vert \Lambda^{\alpha}e^{n+1} \Vert_2^2 + \Vert \Lambda^{\alpha}e^n \Vert_2^2 \right)+ \left(C_4+C_5+C_6\right)\left( \tau^2 + h_x^2 + h_y^2 \right)^2.
\end{align}
{Furthermore, let $\tau_3 = \frac{1}{32C}$. If we take $\tau< \tau_3$, then it holds that $ 1/2-8C\tau \geq 1/4 $}. And since
\begin{align}
\qquad\qquad&\left( \dfrac{1}{2}-8C\tau \right)F_{n+1}-F_1 \nonumber\\
\qquad\qquad\leqslant &F_{n+1}-F_1 -\left(8C\tau + \dfrac{1}{2}\right)\left(\Vert(\mathcal{L}^{\alpha}_{x}+\mathcal{L}^{\alpha}_{y})e^{n+1}\Vert_2^2 + \Vert(\mathcal{L}^{\alpha}_{x}+\mathcal{L}^{\alpha}_{y})e^{n}\Vert_2^2\right)\nonumber\\
           \qquad\qquad &- 8C\tau \left( \Vert \Lambda^{\alpha}e^{n+1} \Vert_2^2 + \Vert \Lambda^{\alpha}e^n \Vert_2^2 \right) \nonumber \\
\qquad\qquad\leqslant &\left( 40C\tau + 20\tau \right)\sum_{k=1}^{n}F_k + \left(C_4+C_5+C_6\right)\left( \tau^2 + h_x^2 + h_y^2 \right)^2,
\end{align}
then we have an estimation for $F_{n+1}$:
\begin{align}\label{estimate458}
F_{n+1}\leqslant 4F_1 + 4\left( 40C + 20 \right)\tau\sum_{k=1}^{n}F_k + 4\left(C_4+C_5+C_6\right)\left( \tau^2 + h_x^2 + h_y^2 \right)^2.
\end{align}
As for the rest terms, from the results of Lemma~\ref{Lemma8},
\begin{align}
  F_1&=\|(\mathcal{L}^{\alpha}_x+\mathcal{L}^{\alpha}_y)e^{1}\|_2^2+\|\Lambda^{\alpha} e^{1}\|_2^2 \leqslant 2(b-a)(d-c)C_e^2(\tau^2+ h_x^2 + h_y^2)^2.\label{e1est}
\end{align}
Substituting (\ref{e1est}) into (\ref{estimate458}) gives
\begin{align}%\label{estimate70}
F_{n+1}\leqslant (160C+80)\tau \sum^{n}_{k=1}F_k+C_7(\tau^2+h_x^2+h_y^2)^2,\notag
\end{align}
where $C_7$ is the sum of all coefficients of $ (\tau^2+h_x^2+h_y^2)^2 $ in the resulting estimation. Thus by the discrete Gronwall's equation, it yields
\begin{align}\label{result2}
F_{n+1} \leqslant &\; C_7\exp\left\lbrace \sum_{k=1}^{n}(160C+80)\tau \right\rbrace(\tau^2+h_x^2+h_y^2)^2 \nonumber\\
                                \leqslant&\;  C_7\exp\left\lbrace (160C+80)T \right\rbrace(\tau^2+h_x^2+h_y^2)^2:=C_8(\tau^2+h_x^2+h_y^2)^2.
\end{align}
Now we combine the consequences (\ref{result1}) and (\ref{result2}). Then it holds that
\begin{align}
\qquad\qquad &\| e^{n+1}\|_2^2 + \| \Lambda^{\alpha} e^{n+1}\|_2^2 + \| (\mathcal{L}^{\alpha}_x+\mathcal{L}^{\alpha}_y)e^{n+1}\|_2^2 \nonumber\\
\qquad\qquad\leqslant &\| e^{n+1}\|_2^2+\| e^{n}\|_2^2 + \| \Lambda^{\alpha} e^{n+1}\|_2^2+\| \Lambda^{\alpha} e^{n}\|_2^2+\| (\mathcal{L}^{\alpha}_x+\mathcal{L}^{\alpha}_y)e^{n+1}\|_2^2+\| (\mathcal{L}^{\alpha}_x+\mathcal{L}^{\alpha}_y)e^{n}\|_2^2\nonumber \\
\qquad\qquad = & E_{n+1}+F_{n+1} \nonumber \\
\qquad\qquad \leqslant & (C_3 + C_8)(\tau^2+h_x^2+h_y^2)^2.\notag
\end{align}

From the mean value inequality, we see
\begin{center}
$ \| e^{n+1}\|_2 + \| \Lambda^{\alpha} e^{n+1}\|_2 + \| (\mathcal{L}^{\alpha}_x+\mathcal{L}^{\alpha}_y)e^{n+1}\|_2  \leqslant \sqrt{3} \sqrt{\| e^{n+1}\|_2^2 + \| \Lambda^{\alpha} e^{n+1}\|_2^2 + \| (\mathcal{L}^{\alpha}_x+\mathcal{L}^{\alpha}_y)e^{n+1}\|_2^2} $.
\end{center}
Thus we need to take $C^0=\max\left\{\sqrt{3(C_3+C_8)}, 3\sqrt{(b-a)(d-c)}C_e\right\}$. Since $C_0=C_{\alpha}\left(\pi / 2\right)^{2\alpha}$
$\cdot C^0$, once $C^0$ is fixed, the condition for $\tau_1, h_1$, i.e., $\tau_1^2+2h_1^2<1/C_0$ can be used to determine $\tau_1, h_1$. Therefore, let $\tau_0=\min\left\{\tau_1,\tau_2,\tau_3,1\right\}$, $ h_x^0=\min\left\{h_1,1\right\} $ and $h_y^0=\min\left\{h_1,1\right\}$, then (\ref{inductionproposition}) is valid for $n+1$.
 This completes the mathematical induction, i.e., it holds that
\begin{align}
\| e^{n+1}\|_2 + \| \Lambda^{\alpha} e^{n+1}\|_2 + \| (\mathcal{L}^{\alpha}_x+\mathcal{L}^{\alpha}_y)e^{n+1}\|_2 \leqslant C^0(\tau^2 + h_x^2 + h_y^2), \quad 1\leqslant n\leqslant N,
\end{align}
which also shows that
\begin{align}
\Vert e^n \Vert_{\infty} \leqslant C_{\alpha}\left(\dfrac{\pi}{2}\right)^{2\alpha}C^0(\tau^2 + h_x^2 + h_y^2):=C_0(\tau^2 + h_x^2 + h_y^2), \quad 1\leqslant n\leqslant N.
\end{align}
\end{proof}

\begin{theorem}\label{theorem5}
Suppose that the solution of problem (\ref{IVP1})--(\ref{IVP2}) is smooth enough, then the solution $U^n$ of
the difference scheme (\ref{scheme1})--(\ref{scheme3}) is bounded in the $L^{\infty}$-norm for $\tau<\tau_0$ and $h<h_0$, i.e.,
 \begin{align}
\| U^n\|_{\infty}\leqslant C_*, \quad  1\leqslant n \leqslant N,
\end{align}
where $\tau_0$ and $h_0$ are the same positive constants in Theorem~\ref{theorem3}.
 \end{theorem}
 \begin{proof}
 From Theorem~\ref{theorem3}, we have
 \begin{align}
\| U^n\|_{\infty}&\leqslant \| u^n\|_{\infty}+\| e^n\|_{\infty}\leqslant C_m+C_0(\tau^2+2h^2), \quad  1\leqslant n \leqslant N,\notag
\end{align}
for $\tau<\tau_0,h<h_0$.

Since $C_0(\tau^2+2h^2)<1$ for  $\tau<\tau_0, h<h_0$, we have
 \begin{align}
\| U^n\|_{\infty}\leqslant  C_*, \quad  1\leqslant n \leqslant N,\notag
\end{align}
where $C_*=C_m+1$. This completes the proof.
\end{proof}

\section{Numerical results}\label{section5}
    In this section, we present some numerical results of the proposed difference scheme
    \eqref{Boundary}, \eqref{scheme1}-\eqref{scheme2} and \eqref{scheme3} to support our main theoretical findings.
    In the computation, uniform spatial meshsizes $h_x=h_y=h$ are used.

\begin{example}\label{ex1}
In order to test the accuracy and verify the unconditional stability of the proposed scheme, we consider the following equation with source terms:
\begin{align*}
u_t(x,y,t)={\bf i} L_{\alpha}u(x,y,t)+{\bf i}\vert u(x,y,t)\vert^2 u(x,y,t)+g(x,y,t).
\end{align*}
The initial conditions, and the source term $g(x,y,t)$ are determined by the exact solution
\begin{equation*}
  u(x,y,t) = {\bf i}\sin(t)x^4(2-x)^4y^4(2-y)^4.
\end{equation*}
We note that the solution satisfies homogeneous Dirichlet condition
\begin{align*}
u\big|_{\partial \Omega}=0,\quad where\quad\Omega=[0,2]\times[0,2].
\end{align*}
\end{example}

Tables \ref{tab1} shows the maximum norm errors and corresponding convergence rates for various choices of $\alpha$.
It can be easily seen that the convergence order of the proposed scheme approaches to 2, which is consistent
with Theorem~\ref{theorem3} in the above section.

To show the unconditional stability of the method, we fix $h$ and vary $\tau$, results for $\alpha = 1.2$ and
$\alpha = 1.5$ are plotted in Figure \ref{Fig22}. As we can see
that the results clearly show that the time step is not related to the spatial meshsize, and as the time steps
go to zero, the dominant errors come from the spatial parts.

\begin{example}\label{ex3}
In this example, we compute the practical problem with the initial conditions as follows:
\begin{align*}
u(x,y,0)&=\frac{2}{\sqrt{\pi}}\exp\lbrace -(x^2+y^2)\rbrace, \quad (x,y)\in \Omega \triangleq [-5,5]\times[-5,5].
\end{align*}
Since the initial value $u(x, y, 0)$ exponentially decays to zero when $(x, y)$ is away from the origin,
we impose the homogeneous boundary condition in the computation
\begin{align*}
u(x,y,0)&=0, \quad (x,y)\textrm{ on } \partial\Omega.
\end{align*}
\end{example}

To show the discrete energy conservative property, we choose $\tau = 1/20, h = 1/20$ and present
the discrete energy $E^n$ at different times for different $\alpha$ $(\alpha = 1.2, \alpha = 1.5, \alpha = 1.8)$
in Table \ref{tab11}. It is found that the proposed scheme preserves the  energy conservation property very well.
Therefore, the scheme is suitable for long-time simulation.

Taking $\tau = 1/16$ and $h= 1/16$, we plot the numerical solutions that derived from the proposed scheme to
investigate  the propagation of solution profile by using  different $\alpha$. The results are shown in Figures
\ref{Fig2}-\ref{Fig4}.  One can see that the value of $\alpha$ significantly affects the shape of the solution.
The larger  $\alpha$ is,  the more heterogeneity of the solution at late time stages is.

\section{Concluding remarks}\label{section6}
In this paper, we developed a linearized semi-implicit finite difference scheme for 2D SFNSE.
A rigorous analysis of the proposed finite difference scheme is carried out, which includes the
conservation, the unique solvability, the unconditional stability and the second-order convergence
in the $L^{\infty}$-norm. {What is worth mentioning is that the theoretical conclusions such as the optimal pointwise $L^{\infty}$ error estimate are presented for the first time in the literatures in this field.
In the end, we implemented the difference scheme through two numerical tests, which showed a perfect consistency with our theoretical findings.}

In addition, our numerical method and  the analysis technique of the optimal pointwise error estimates can be
 easily extended to the cases with spatial fourth-order accuracy \cite{ZS2014}, the 2D coupled nonlinear
 space  fractional Schr\"{o}dinger equation \cite{RZ2016},
 2D space  fractional Ginzburg-Landau equation \cite{Pan2020b,ZPLR}
 and some other space  fractional diffusion equation in 2D and 3D~\cite{Yue2019,ZLAT2009,He2020,Zhang2020}, which will be our future work.
 In addition, for the resulting systems of algebraic equations, the coefficient matrices have Toeplitz structure,
 which can be solved by adopting a super-fast solver with preconditioner \cite{LS2013,LMS2017,LMS2019,LMS2018} or multigrid methods \cite{Pan2017,Pan2021} to reduce the
 CPU time and storage requirement in future.

 \section*{Acknowledgement}
{Hongling Hu was supported by the National Natural Science Foundation of China (No. 12071128), the Research Foundation of Education Bureau of Hunan Province of China (No. 18B002). Dongdong He was supported by  the president's fund-research start-up fund from the Chinese University of Hong Kong, Shenzhen (No. PF01000857). Kejia Pan was supported by Science Challenge Project (No. TZ2016002), the National Natural Science Foundation of China (No. 41874086).
Qifeng Zhang was supported by  the Natural Sciences Foundation of Zhejiang Province (No. LY19A010026), the Zhejiang Province ``Yucai'' Project (No. 2019YCGC012) and the Fundamental Research Funds of Zhejiang Sci-Tech University (No. 2019Q072).}

  \section*{Conflict of Interest Statement}

 The authors  declare that they have no conflict of interest.

   \section*{Data Availability Statement}

 The data that support the findings of this study are available from the corresponding author upon reasonable request.
\begin{appendix}
\section{The proof of the mass and energy conservation}\label{appendixa}
 Multiplying \eqref{IVP1} by $u^{\ast}(x,y,t)$ and integrating in $\mathbb{R}^2$, we have
      \begin{equation}\label{attachment1}
        {\bf i} \int_{\mathbb{R}^2}\limits u_t(x,y,t)u^{\ast}(x,y,t)dxdy + \int_{\mathbb{R}^2}\limits L_{\alpha}u(x,y,t)u^{\ast}(x,y,t)dxdy +
        \int_{\mathbb{R}^2}\limits |u(x,y,t)|^4dxdy =0.
      \end{equation}
      We notice that
      \begin{equation*}
        \int_{\mathbb{R}^2}\limits L_{\alpha} u(x,y,t)u^{\ast}(x,y,t)dxdy = \int_{\mathbb{R}^2}\limits |(-\Delta)^{\frac{\alpha}{4}}u(x,y,t)|dxdy,
      \end{equation*}
      and
      \begin{equation*}
       Re \left\{u_t(x,y,t)u^{\ast}(x,y,t)\right\} = \frac{1}{2}(|u(x,y,t)|^2)_t.
      \end{equation*}
      Taking imaginary part of \eqref{attachment1}, we have
      \begin{equation*}
        \frac{d}{dt}\int_{\mathbb{R}^2}\limits \frac{1}{2}|u(x,y,t)|^2dxdy=0.
      \end{equation*}
      It implies that the first equation holds in \eqref{QE1}.

      Multiplying \eqref{attachment1} by $-u^{\ast}_t(x,y,t)$ and integrating in $\mathbb{R}^2$, we have
      \begin{equation}\label{attachment2}
        -{\bf i} \int_{\mathbb{R}^2}\limits |u_t(x,y,t)|^2 dxdy -  \int_{\mathbb{R}^2}\limits L_{\alpha}u(x,y,t)u^{\ast}_t(x,y,t)dxdy
        -  \int_{\mathbb{R}^2}\limits |u(x,y,t)|^2u(x,y,t)u^{\ast}_t(x,y,t)dxdy = 0.
      \end{equation}
          We notice that
     \begin{equation}\label{attachment3}
        - \int_{\mathbb{R}^2}\limits L_{\alpha}u(x,y,t)u^{\ast}_t(x,y,t)dxdy = \int_{\mathbb{R}^2}\limits (-\Delta)^{\frac{\alpha}{4}}u(x,y,t)\frac{\partial}{\partial t}\left((-\Delta)^{\frac{\alpha}{4}}u(x,y,t)\right)dxdy.
     \end{equation}
     Taking real part of \eqref{attachment3}, we have
     \begin{align*}
       - Re \left\{ L_{\alpha} u(x,y,t) u^{\ast}_t(x,y,t)dxdy \right\}=&
       \displaystyle \int_{\mathbb{R}^2}\limits  \frac{d}{dt}\left(\frac{1}{2} |(-\Delta)^{\frac{\alpha}{4}}u(x,y,t)|^2\right) dxdy\\
       =&  \displaystyle  \frac{1}{2}\frac{d}{dt} \int_{\mathbb{R}^2}\limits  |(-\Delta)^{\frac{\alpha}{4}}u(x,y,t)|^2dxdy.
       \end{align*}
     Taking real part of \eqref{attachment2}, we have
     \begin{align*}
        \frac{1}{2}\frac{d}{dt} \int_{\mathbb{R}^2}\limits  |(-\Delta)^{\frac{\alpha}{4}}u(x,y,t)|^2dxdy -\frac{1}{4}\frac{d}{dt}\int_{\mathbb{R}^2}\limits |u(x,y,t)|^4 dxdy=0,
     \end{align*}
     which implies that the second equation holds in \eqref{QE1}.
\end{appendix}



\newpage

\begin{table}[!h]
\tabcolsep=5pt
\caption{$L^\infty$-norm errors and their convergence orders for Example \ref{ex1}.}\label{tab1}
 \centering
  \begin{tabular}{lllllllll}
    \hline
   \multirow{1}*{$\tau=h$}   & \multicolumn{2}{l}{$\alpha=1.2$} &   \multicolumn{2}{l}{$\alpha=1.5$} &  \multicolumn{2}{l}{$\alpha=1.9$}  & \multicolumn{2}{l}{$\alpha=2.0$}  \\
    \cline{2-3} \cline{4-5} \cline{6-7} \cline{8-9} &   $\| u^n-{U}^n\|_\infty$ &  order  &  $\| u^n-{U}^n\|_\infty$ &  order      &$\| u^n-{U}^n\|_\infty$ & order   &    $\| u^n-{U}^n\|_\infty$  & order   \\
\hline
$1/16$  &   4.62e$-$02   &   -    &   3.01e$-$02  &    -      &  4.70e$-$02  &    -    &   2.08e$-$02   &   -           \\
$1/32$  &   9.67e$-$03   &   2.26 &   1.05e$-$02  &    1.52   &  7.60e$-$03  &    2.63 &   5.55e$-$03   &   1.91        \\
$1/64$  &   1.81e$-$03   &   2.42 &   2.29e$-$03  &    2.19   &  1.55e$-$03  &    2.30 &   1.27e$-$03   &   2.13        \\
$1/128$ &   5.48e$-$04   &   1.72 &   5.24e$-$04  &    2.13   &  3.89e$-$04  &    1.99 &   3.07e$-$04   &   2.05        \\
$1/256$ &   1.21e$-$04   &   2.18 &   1.26e$-$04  &    2.06   &  1.08e$-$04  &    1.85 &   7.65e$-$05   &   2.00        \\
\hline
\end{tabular}
\end{table}

\begin{table}[!h]
\tabcolsep=15pt
\caption{The  discrete energy $E^n$ at different times with $\tau=1/20$ and $h=1/20$ for Example \ref{ex3}.}\label{tab11}
 \centering
  \begin{tabular}{llll}
    \hline
  & $\alpha = 1.2 $   & $\alpha = 1.5 $   & $\alpha = 1.8 $ \\
\hline
$t=0.5$ & 2.6252539243441122 &   2.8186429624622438  &  3.1334617831273350   \\
$t=1.0$ & 2.6252539243434061 &   2.8186429624639984  &  3.1334617831271649   \\
$t=1.5$ & 2.6252539243427715 &   2.8186429624659994  &  3.1334617831260112   \\
$t=2.0$ & 2.6252539243406749 &   2.8186429624668641  &  3.1334617831288387   \\
$t=2.5$ & 2.6252539243384239 &   2.8186429624676719  &  3.1334617831306835   \\
$t=3.0$ & 2.6252539243365343 &   2.8186429624698901  &  3.1334617831308118   \\
$t=3.5$ & 2.6252539243381099 &   2.8186429624709093  &  3.1334617831335723   \\
$t=4.0$ & 2.6252539243396047 &   2.8186429624729565  &  3.1334617831349845   \\
$t=4.5$ & 2.6252539243401332 &   2.8186429624734930  &  3.1334617831369416   \\
$t=5.0$ & 2.6252539243413202 &   2.8186429624735556  &  3.1334617831445417   \\
\hline
\end{tabular}
\end{table}

\newpage

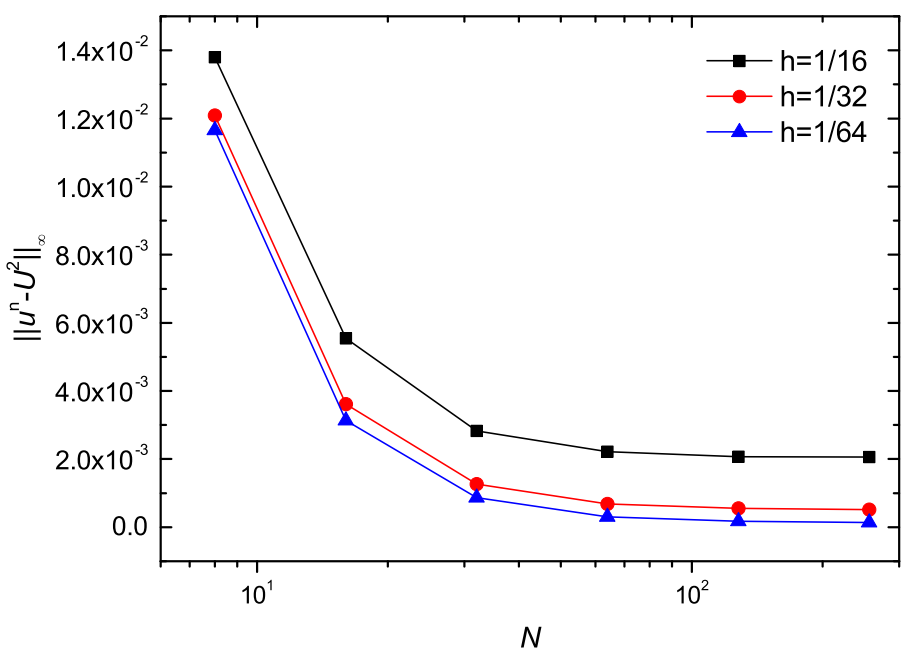
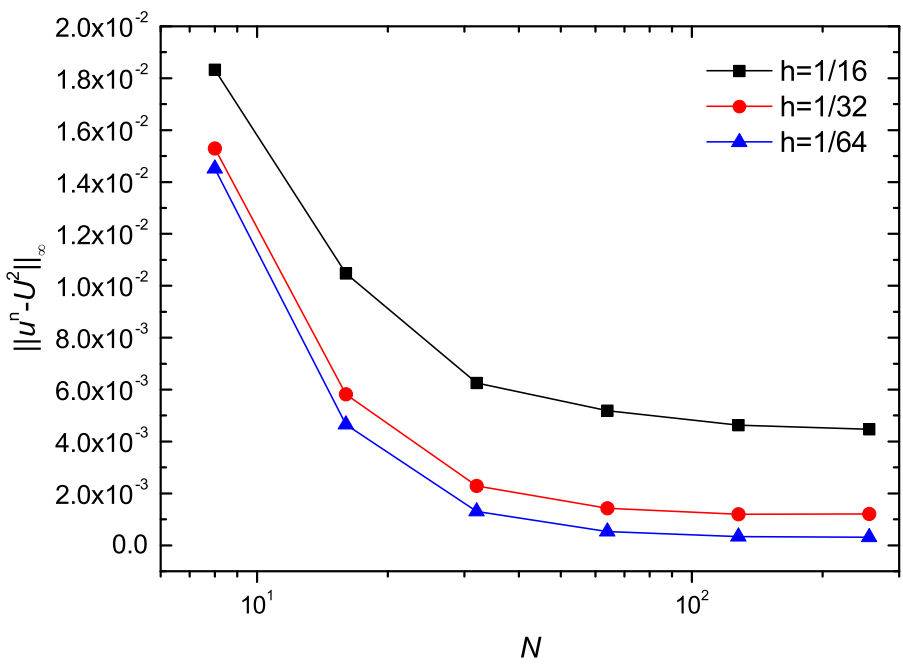
\end{document}